\documentclass[9pt]{amsart}
\usepackage{amsmath, amsthm, amscd, amsfonts,graphicx}
\usepackage{amsmath,amsthm,amsthm, amscd, amsfonts, stackrel, latexsym, dsfont,amssymb,mathrsfs,textcomp,wasysym,hyperref}
\usepackage{tikz}
\usetikzlibrary{matrix,arrows}
\usepackage[all]{xy}

\usepackage{color} 
\definecolor{slateblue}{rgb}{0,0.0,0.8}

\usepackage{xcolor}
\usepackage{graphicx}
\hypersetup{%
	colorlinks=true,
	linkcolor=slateblue,
	linkbordercolor=gray,
	citecolor=slateblue
}

\theoremstyle{definition}
\newtheorem{definition}{Definition}[section]
\newtheorem{example}[definition]{Example}
\theoremstyle{remark}
\newtheorem{remark}[definition]{Remark}
 
\theoremstyle{plain}
\newtheorem{theorem}[definition]{Theorem}
\newtheorem{lemma}[definition]{Lemma}
\newtheorem{proposition}[definition]{Proposition}
\newtheorem{corollary}[definition]{Corollary}
\newtheorem{notation}[definition]{Notation}
\usepackage{mathptmx}

\makeatletter
\@namedef{subjclassname@2020}{\textup{2020} Mathematics Subject Classification}
\makeatother

\usepackage{amsthm}

\newtheorem{innercustomgeneric}{\customgenericname}
\providecommand{\customgenericname}{}
\newcommand{\newcustomtheorem}[2]{%
	\newenvironment{#1}[1]
	{%
		\renewcommand\customgenericname{#2}%
		\renewcommand\theinnercustomgeneric{##1}%
		\innercustomgeneric
	}
	{\endinnercustomgeneric}
}

\newcustomtheorem{customthm}{Proposition}
\newcustomtheorem{customlemma}{Lemma}

\begin{document}

	\title[Submersions, immersions, and \'{e}tale maps in diffeology]{Submersions, immersions, and \'{e}tale maps in diffeology}
	
	\author[A. Ahmadi]{Alireza Ahmadi}
	    \address{Alireza Ahmadi, Department of Mathematical Sciences, Yazd University, 89195--741, Yazd, Iran }
	\email{ahmadi@stu.yazd.ac.ir, alirezaahmadi13@yahoo.com}

	

	\maketitle

	\begin{abstract}
	Although structural maps such as subductions and inductions appear naturally in diffeology, 		
	one of the challenges is providing suitable analogous for submersions, immersions, and \'{e}tale maps (i.e., local diffeomorphisms) consistent with the classical versions of these maps between manifolds. In this paper, we consider diffeological submersions, immersions,  and \'{e}tale maps as an adaptation of these maps to diffeology by a nonlinear approach. In the case of manifolds, there is no difference between the classical and diffeological versions of these maps. 
	Moreover, we study their diffeological properties from different aspects in a systematic fashion with respect to the germs of plots. 
	We also discuss notions of embeddings of diffeological spaces and regard diffeological embeddings similar to those of manifolds. In particular, we show that diffeological embeddings are inductions. 
	
	In order to characterize the considered maps from their linear behaviors, we introduce a class of diffeological spaces, so-called diffeological \'etale manifolds, which not only contains the usual manifolds but also includes irrational tori. We state and prove versions of the rank and implicit function theorems, as well as the fundamental theorem on flows in this class. As an application, we use the results of this work to facilitate the computations of the internal tangent spaces and diffeological dimensions in a few interesting cases.
	\end{abstract}
	
	\tableofcontents
	\section{Introduction}\label{S1} 
	
	Manifolds 
	are nice regular spaces that crop up throughout mathematics and science, however, the category of these objects does not behave well under subspaces, quotients, and advanced constructions related to function spaces. This context is also inadequate to afford spaces with singularities such as orbifolds and quasifolds (see, e.g., \cite{IKZ,IP}).  A treatment for these difficulties is extending the category of manifolds to that of diffeological spaces.
	Formally, a diffeological space is defined by a set $ X $  endowed with a structure of \textit{diffeology} on it, which is a collection of  parametrizations into $ X $ defined on  
	open subsets of Euclidean
	spaces, called \textit{plots}, regarded as smooth parametrizations by enjoying
	the axioms of \textit{covering}, \textit{smooth compatibility} and \textit{locality}.  
	Diffeology enables us to think of smooth maps between diffeological spaces as maps preserving plots under composition---a property compatible with smooth maps between manifolds.
	These spaces and smooth maps between them constitute a rich category---a complete, cocomplete, and Cartesian closed quasitopos---and provide a convenient framework to do differential geometry (see \cite{BH,PIZ}). 
	Diffeology was established by J.-M. Souriau \cite{JMS} in the early 1980s and has been developed by P. Iglesias-Zemmour and others, in various directions of the theory and applications.
	The main reference for this theory is the book \cite{PIZ}.

	An inspiring and favorite example of diffeological spaces is the irrational torus 
	$ \mathbb{T}_{\alpha}=\mathbb{R}/(\mathbb{Z}+\alpha\mathbb{Z}) $, $ \alpha\notin\mathbb{Q}  $, which is neither a manifold nor an orbifold.
	A full classification and some features of 
	irrational tori have been studied in \cite{DI, PIZ85}.
	In this paper, we view the quotient map $ \mathbb{R}\rightarrow\mathbb{T}_{\alpha} $ from another perspective, which motivates us to introduce diffeological \'{e}tale maps. Since the D-topology of $ \mathbb{T}_{\alpha} $ is trivial  (see \cite[p. 56]{PIZ}), this map cannot be \'{e}tale\footnote{By an \'{e}tale map, we mean a local diffeomorphism with respect to the D-topology (see  \cite[\S 2.5]{PIZ}), and an \'{e}tale space is the source of an \'{e}tale map. This terminology is better suited to the generalization we consider for it.}. 
	However, its pullbacks along the plots in  $ \mathbb{T}_{\alpha} $ are all \'{e}tale.
	Abstracting this property, we define diffeological \'{e}tale spaces and maps, formulated 
	with respect to plots instead of D-open\footnote{Throughout this paper, the prefix D of a topological property indicates the same property in terms of the D-topology.} subsets (Definition \ref{def-deta}). 
	More generally, we can say that a covering map of diffeological spaces is nothing but a map which is both a diffeological \'{e}tale map and a diffeological fiber bundle, just like the classical situation. Furthermore, it is shown that 
	on manifolds, there is no difference between \'{e}tale spaces and diffeological \'{e}tale spaces.
	And any diffeological \'{e}tale space on a diffeological orbifold is again a  diffeological orbifold (Theorem \ref{the-orb}).
		
	The same approach can be applied to
	address the problem of finding suitable analogous for submersions and immersions in diffeology such that respect the desirable properties of the classical version of these maps 
	and include diffeological \'{e}tale maps entirely in their intersection.
	Although there are distinguished maps such as subductions and inductions in diffeology, which are actually strong epimorphisms and monomorphisms (see \cite[Propositions 34 and 37]{BH}),
	when we restrict ourselves to manifolds, or even open subsets of Euclidean
	spaces, they do not play the role of submersions and immersions (see Example \ref{exa-sub-sum}, also \cite{KMW} for a detailed discussion on inductions between manifolds).

	For this purpose, we first define submersions and immersions of diffeological spaces according to the nonlinear characterizations of these maps between manifolds. Then, we propose diffeological submersions and immersions by smooth maps whose pullbacks along plots are submersions and immersions, respectively.
	Of course, a diffeological submersion is nothing more than a \textit{strict} local subduction in \cite{PIZ}, that is, a local subduction onto its image. 
	As expected, diffeological submersions and immersions on manifolds are the same as the usual submersions and immersions (Corollaries \ref{cor-dsub-mfd} and \ref{cor-dimm-mfd}).
	
	We study the key properties of diffeological submersions, immersions, and \'etale maps from different aspects.
	As mentioned above, a smooth map takes plots to plots and gives rise to a natural equation of plots.
	Hence the following descriptions for these maps are obtained:
\begin{customthm}{A}\label{pi1}
		Let $ f:X\rightarrow Y $ be a smooth map between  diffeological spaces and consider the equation  $ f\circ L=P $ on the 
	\textit{germs} of plots, for given $ x\in X $, any plot $ P:U\rightarrow Y $, and each $ r\in U $ with $ P(r)=f(x)  $. 
		\begin{enumerate}
		\item[$ \textbf{S.}\quad $] 
		$ f $ is a diffeological submersion if and only if for any given data as above, there exists \textit{at least} one local lift plot $ L:V\rightarrow X $ defined on an open
		neighborhood $ V\subseteq U $ of $ r $ such that $ f\circ L=P|_V $ and $ L(r)=x $.	
		
		\item[$ \textbf{I.}\quad $] 
		If $ f $ is a diffeological immersion, then for any given data as above, there exists \textit{at most} one local lift plot $ L:V\rightarrow X $ defined on an open
		neighborhood $ V\subseteq U $ of $ r $ such that $ f\circ L=P|_V $ and $ L(r)=x $.
		
		\item[$ \textbf{E.}\quad $] 
		$ f $ is a diffeological \'etale map if and only if for any given data as above, there exists a \textit{unique}\footnote{Unique up to germs of plots}  local lift plot $ L:V\rightarrow X $ defined on an open neighborhood $ V\subseteq U $ of $ r $ such that $ f\circ L=P|_V $ and $ L(r)=x $.
	\end{enumerate}
\end{customthm}
	While $ \textbf{S} $ and $ \textbf{E} $ are equivalences, notice that $ \textbf{I} $ is only an implication. For the sake of completeness, we add diffeological injections to the discussion as those maps that turn $ \textbf{I} $ into an equivalence, which work as an auxiliary tool for us.
	
Another  diffeological property possessed by these kinds of maps is that the diffeological dimensions (see \cite{PIZ2007}) of the source and target spaces of diffeological submersions, immersions,  and \'{e}tale maps are comparable. 
\begin{customthm}{B}
	Let $ f:X\rightarrow Y $ be a smooth map between  diffeological spaces.
		\begin{enumerate}
		\item[$ \textbf{S}'\textbf{.}\quad $] 
		If $ f:X\rightarrow Y $ is a diffeological submersion, then
		$ \mathrm{dim}(X)\geq\mathrm{dim}(f(X)) $.
		
		\item[$ \textbf{I}'\textbf{.}\quad $] 
		If $ f:X\rightarrow Y $ is a diffeological immersion, then
		$ \mathrm{dim}(X)\leq\mathrm{dim}(Y) $.
		
		\item[$ \textbf{E}'\textbf{.}\quad $] 
		If	$ f:X\rightarrow Y $ is a diffeological \'etale map, then
		$ \mathrm{dim}(X)=\mathrm{dim}(f(X))\leq\mathrm{dim}(Y)  $.
	\end{enumerate}
	(Here $ \mathrm{dim} $ indicates the diffeological dimension).
\end{customthm}

This result  extends the corresponding result in the classical setting of manifolds, which is satisfactory.
In particular, $ \textbf{E}' $ is the best that can be proved, even for \'etale maps (see Remark \ref{rem-etal-dim}).

	To study diffeological spaces from a linear-algebraic viewpoint, the internal tangent spaces and bundles are defined and investigated thoroughly in \cite{CW,Hec,HM-V} besides the external types in the literature.
	In this regard, the following results are achieved, which may be useful  to compute the internal tangent spaces (see Example \ref{exa-irt}):

\begin{customthm}{C} \label{pi2}
	Let $ f:X\rightarrow Y $ be a smooth map between  diffeological spaces.
	\begin{enumerate}
		\item[$ \textbf{S}''\textbf{.}\quad  $] 
		If $ f:X\rightarrow Y $ is a diffeological submersion, then the internal tangent map $ df_x:T_xX\rightarrow T_{f(x)}Y $ is an epimorphism, at each point of $ x\in X $.
		
		\item[$ \textbf{I}''\textbf{.} \quad $] 
		If $ f:X\rightarrow Y $ is a diffeological immersion, then the internal tangent map $ df_x:T_xX\rightarrow T_{f(x)}Y $ is a monomorphism, at each point of $ x\in X $.
		
		\item[$ \textbf{E}''\textbf{.}\quad  $] 
		If $ f:X\rightarrow Y $ is a diffeological \'etale map, then the internal tangent map $ df_x:T_xX\rightarrow T_{f(x)}Y $ is an isomorphism, at each point of $ x\in X $.
	\end{enumerate}  
	Recall that the internal tangent map of a smooth map is the linear map between the internal tangent spaces induced by universal property.
\end{customthm}
	However, 
	the converse to Proposition \ref{pi2} does not hold (see Remark \ref{rem-tng}).
	In fact, on more complicated spaces than manifolds, one might lose the linear characterizations of the classical version of the  maps mentioned above, such as the inverse functions theorem.
	In light of \cite[Example 3.23 and Propositions 3.13]{CW}, we conclude nothing for the external tangent maps.
    This justifies why we have employed a nonlinear approach for our purposes in the first place. 
		
	To avoid pathological cases such as these in studying the linear-algebraic properties, we restrict ourselves to a somewhat more regular class of diffeological spaces called \textit{diffeological \'etale manifolds}, which are generated by diffeological \'etale maps defined on open subsets of a fixed Euclidean space as charts.
	We observe that for smooth maps between diffeological \'etale manifolds, the converse to $ \textbf{S}'',\textbf{I}'' $ and $ \textbf{E}'' $ is true (Theorem \ref{the-IFT}).
	This class of diffeological spaces contains not only the usual manifolds but also irrational tori. 
	Even without having locally injective charts, diffeological \'etale manifolds behave still similar to the usual ones in many aspects.
	We can particularly take advantage of the rank and implicit function theorems (Theorems \ref{the-Rank} and \ref{the-ImFT}).
	Using these tools, some results of the classical geometry of manifolds can be extended to diffeological \'etale manifolds. 
	Furthermore, we discuss the interplay between vector fields and flows in diffeology, and prove the fundamental theorem of flows for diffeological \'etale manifolds (Theorem \ref{the-flow}). 
	As an application, we generalize a result in \cite{CW,Hec,HM-V} and compute the internal tangent spaces of the diffeomorphism group of a D-compact diffeological \'etale manifold, especially an irrational torus, which are isomorphic to the space of its vector fields (Corollary \ref{coro-diff-comp}).
	
	Apart from the diffeological structure itself, 
	every diffeological space has other information layers, including the internal tangent spaces and the D-topology, which are essentially by-products of the diffeological structure.
	Accordingly, one has various choices to think of a notion of embeddings for diffeological spaces.
	Embeddings in the sense of Iglesias-Zemmour are those maps that are both an induction and a D-topological embedding.
	Following \cite{KMW}, one can take  weak embeddings as those maps that are both an induction and a diffeological immersion.
	Analogous to embeddings of manifolds, but
	we think of diffeological embedding or \textit{strong} embeddings as those maps that are both a diffeological immersion and a D-topological embedding. 
	It is proved that diffeological embedding are exactly those maps that are both an embedding and a weak embedding (Corollary \ref{coro-embd-st}).
	In other words, a map is a diffeological embedding if and only if it is an induction, a diffeological immersion, and a D-embedding.
	Hence we can say a diffeological embedding embeds all of these information layers.
	Examples of diffeological embeddings are abundant.
	Besides embeddings of manifolds,
	the natural embeddings of a diffeological space into its path space and into its power set\footnote{Equipped with the power set diffeology in \cite{DP} and some other diffeologies studied in \cite{AM}.}, and
	those arise from smooth retractions of diffeological spaces provide interesting examples of diffeological embeddings.
 The latter may  be helpful to determine the diffeological dimension in some cases  (see Examples \ref{exa-dim-R} and \ref{exa-dim}).

	 \paragraph{\textbf{Organization.}} The organization of the paper is as follows. 
	Section \ref{S2} is devoted to an overview of the necessary material from diffeological spaces, especially the pullbacks of smooth maps and the internal tangent spaces and tangent bundles.
	In Section \ref{S3}, after a mention of some concepts close to subductions such as strong, global, and weak subductions, we introduce submersions and diffeological submersions, explore their main features and connections with submersions of manifolds.
	In Section \ref{S4},  we first investigate  diffeological injections, next proceed with the study of immersions and diffeological immersions, compare them with immersions of manifolds.
	We also consider diffeological embeddings and give several examples.
	In Section \ref{S5},
	we deal with diffeological \'{e}tale spaces and maps, describe their characterizations, and examine diffeological \'{e}tale spaces over diffeological orbifolds.
	We conclude the paper in Section \ref{S6} with a detailed study of diffeological \'{e}tale manifolds. We state and prove versions of the rank and implicit function theorems, as well as the fundamental theorem on flows for them.



	\section{Background on diffeologies}\label{S2}
	We  recall some fundamental definitions and constructions of diffeological spaces from \cite{PIZ}, which will be encountered in the rest of this paper.
	\subsection{Basic definitions}
	\begin{definition}
		An $n$-\textbf{domain}, for a nonnegative integer $n$, is an open subset of  Euclidean space $\mathbb{R}^n$ endowed with the standard topology. 
		Any map from a domain to a set $X$ is said to be a \textbf{parametrization} in $X$.  
		If the domain of definition of a parametrization $ P $, denoted by 
		$ \mathrm{dom}(P) $, is an $n$-domain, then $ P $ is called an $n$-parametrization. 
		The only $0$-parametrization with the value $x\in X$ is denoted by the bold letter $ \mathbf{x}$.
		A parametrization $ P $ in $ X $ with $ 0\in\mathrm{dom}(P) $ and $ P(0)=x $ is called a parametrization \textbf{centered} at $ x $.
		A family $\lbrace P_i: U_i\rightarrow X\rbrace_{i\in J}$ of $n$-parametrizations is \textbf{compatible} if
		$P_i|_{U_i\cap U_j}=P_j|_{U_i\cap U_j}$, for all $i, j\in J$. 
		For such a family, the parametrization 
		$P:\bigcup_{i\in J} U_i\rightarrow X$ given by $P(r)=P_i(r)$ for $r\in U_i$, 
		is said to be the \textbf{supremum} of the family. 
		By convention, the supremum of the empty family is the empty parametrization $ \varnothing\rightarrow X $.
	\end{definition} 

	\begin{definition} 
		A \textbf{diffeology} $ \mathcal{D} $ on a set $ X $ is a collection of parametrizations in $ X $ that fulfills the following axioms:
		\begin{enumerate}
			\item[\textbf{D1.}] 
			The union of the images of the elements of $ \mathcal{D} $ covers $ X $.
			
			\item[\textbf{D2.}] 
			For every element $P:U\rightarrow X$ of $\mathcal{D}$ and every smooth map $F:V\rightarrow U$ between domains, the parametrization $P\circ F$ belongs to $\mathcal{D}$.
			\item[\textbf{D3.}]
			The supremum of any compatible family of elements of $\mathcal{D}$ belongs to $\mathcal{D}$.
		\end{enumerate}
		A \textbf{diffeological space} $ (X,\mathcal{D}) $ is an underlying set $X$ together with a
		diffeology $\mathcal{D}$, whose elements are called the \textbf{plots} in $ X $.
		A diffeological space is just denoted by the underlying set, when the diffeology is understood.
	\end{definition} 
	
	By the axioms of diffeology, in every diffeological space, the locally constant
	parametrizations are plots.
	In particular, the empty parametrization is a plot.
	We also take into account the empty set $ \varnothing $ as a diffeological space with the diffeology $ \lbrace\varnothing\rightarrow\varnothing\rbrace $.
	\begin{example}
		Let $ X $ be any set.
		The set of the locally constant parametrizations in  $ X $ is a diffeology on $ X $ called the \textbf{discrete
			diffeology}.
		The set of all parametrizations in $ X $ is a diffeology on $ X $ called the \textbf{indiscrete} or \textbf{coarse
			diffeology}.
	\end{example} 
	
	As the following examples show, the axioms of diffeology are independent.

	\begin{example}
		\begin{enumerate}
			\item[(i)] 
			Consider the set $ \mathbb{R} $, and the collection of bounded smooth parametrizations $ P:U\rightarrow\mathbb{R} $.
			This collection satisfies  D1  and  D2, but not D3.
			
			\item[(ii)] 
			The set of smooth $ 1 $-parametrizations
			$ 	P:U\rightarrow\mathbb{R}^n $ in a Euclidean space $ \mathbb{R}^n  $ obeys D1 and D3, but not D2.
			\item[(iii)] 
			Let $ X=\{0,1\} $.  The set of all constant parametrizations in $ X $ with the value $ 0 $ meets D2 and D3, but not D1.
		\end{enumerate}
		
	\end{example}
	
	\begin{definition} (\cite[Definition 2.6]{DA}).
		A collection $\mathcal{D}$ of parametrizations in a set $X$ satisfying D1 and D2 is said to be a  \textbf{prediffeology} on $X$.
		If $\mathcal{D}$ fulfills D1, it is a \textbf{parametrized cover} of $X$.
	\end{definition}  
 
	\begin{definition}
		Let  $X$ be a set and let $\mathcal{C}$ be a parametrized cover of $X$.
		The \textbf{prediffeology generated} by $\mathcal{C}$, which we denote by $\lfloor\mathcal{C}\rfloor$, consists of the composite parametrizations $ P\circ F $ in which $ P $ is an element of $\mathcal{C}$ and $ F $ is a smooth map between domains.
			The \textbf{diffeology generated} by $ \mathcal{C} $, which we denote by $\langle\mathcal{C}\rangle$, is the set of  parametrizations $ P $ 
		that are the supremum of a compatible family
		$ \lbrace P_i\rbrace_{i\in J} $ of parametrizations in $ X $ with $ P_i\in\lfloor\mathcal{C}\rfloor $.
		For a diffeological space $ (X,\mathcal{D}) $,
		a \textbf{covering generating family} is a parametrized cover $\mathcal{C}$ of $ X $  generating the diffeology of the space, i.e., $\langle\mathcal{C}\rangle=\mathcal{D}$.
		Denote by $ \mathsf{CGF}(X) $ the collection of all covering generating families of the space $ X $.
	\end{definition}
	
	\begin{example}
		For any diffeological space $ X $, each of the following collections is a covering generating family:
		\begin{enumerate}
			\item[$ {\tiny\blacktriangleright}  $] 
			The diffeology of the space,
			\item[$ {\tiny\blacktriangleright}  $] 
			The collection of round plots, i.e., plots  whose domains are open balls,
			\item[$ {\tiny\blacktriangleright}  $] 
			The collection of global plots $ \mathbb{R}^n\rightarrow X $ ($n$ ranges over
			nonnegative integers),
			\item[$ {\tiny\blacktriangleright}  $] 
			The collection of centered plots, i.e., plots $ U\rightarrow X $ with $ 0\in U $.
		\end{enumerate} 
	\end{example} 
	
	\begin{definition}
		The \textbf{dimension}  of a  diffeological space $ X $, denoted by $ \mathrm{dim}(X) $, is a nonnegative integer that is defined by
		\begin{center}
			$ \mathrm{dim}(X)= ~~~\inf_{\mathcal{G} \in\mathsf{CGF}(X)} ~~\mathrm{dim}(\mathcal{G}), $
		\end{center}
		where
		\begin{center}
			$ \mathrm{dim}(\mathcal{G})=~~~\sup_{P\in\mathcal{G}} ~~\mathrm{dim}\big{(}\mathrm{dom}(P)\big{)} $
		\end{center}
		for each $ \mathcal{G} \in\mathsf{CGF}(X) $.
		We write $ \mathrm{dim}(X)=\infty $ if $ X $ has no covering generating family of finite dimension. 
	\end{definition} 
	
	\begin{remark}
		In the above definition, notice that the infimum is taken over all covering generating families, not  generating families.
		But this is equivalent to \cite[\S 1.78]{PIZ}, because one can always replace any generating family by a covering generating family with the same dimension (see \cite[\S 1.72]{PIZ}).
	\end{remark}

	\begin{definition}
		Let $X$ and $Y$ be two diffeological spaces. A map $f:X\rightarrow Y$ is \textbf{smooth} if for every plot $P$ in $X$, the composition $f\circ P$ is a plot in the space $Y$.
			The set of all smooth maps from $X$ to $Y$  is denoted by $\mathrm{C}^{\infty}(X,Y)$.
			Denote by $ \mathsf{Diff} $ the category of diffeological spaces and smooth maps.
			The isomorphisms in the category $ \mathsf{Diff} $ are called \textbf{diffeomorphisms}.
	\end{definition} 
	
	Plots in a diffeological space are exactly the smooth parametrizations in the space.

	\begin{example}
		Any orbifold, especially a manifold, has a standard diffeology generated by its atlas.
		Endowing manifolds with the standard diffeology, a map between manifolds is smooth in the usual sense if and only if it is smooth in the diffeological sense (see \cite[\S 4.3]{PIZ}). 
	\end{example} 
	
	\begin{example}\label{exa-lasa}
		Denote by $ \mathbb{R}_k^n $ the diffeological space with the underlying set $ \mathbb{R}^n $ and the diffeology generated by the usual smooth $ k $-parametrizations in $ \mathbb{R}^n $.
		Then $ \mathrm{id}_{\mathbb{R}^n}:\mathbb{R}_k^n\rightarrow\mathbb{R}^n $ is a bijective smooth map. If $ k< n $, it is not a diffeomorphism.
	\end{example}

	\begin{definition}
		Every diffeological space $ X $ has a natural topology called the D-\textbf{topology} in which a subset of $X$ is D-\textbf{open} if its preimage by any plot is open. 
		We denote a diffeological space $ X $ regarded as the topological space endowed with the D-topology by $ D(X) $.
	\end{definition}
	
	Any smooth map is D-continuous, that is, continuous with respect to the D-topology \cite[\S 2.9]{PIZ}.

	\begin{proposition}
		The D-topology of a usual orbifold  $ \mathcal{O} $, especially a manifold, coincides with the underlying topology.
	\end{proposition}
	\begin{proof}
		The underlying topology is contained in the D-topology, because orbifold charts are  continuous with respect the underlying topology. 
		On the other hand, let $ O $ be a D-open subset of $ \mathcal{O} $ and $ x\in O $. 
		Take an orbifold chart $ \varphi:U\rightarrow\mathcal{O} $ with $ x\in \varphi(U) $.
		Since orbifold charts are open maps, $ \varphi\circ\varphi^{-1}(O)\subseteq O $ is an open subset of $ \mathcal{O} $ containing $ x $.
		This follows that the D-open subset $ O $ itself is an open subset of $ \mathcal{O} $ with respect to the underlying topology.
	\end{proof}
	
	\subsection{Constructions associated with diffeological spaces}
	
	\begin{definition}
		A map $ f:X\rightarrow Y $ between diffeological spaces is an \textbf{induction} if it is injective and the pullback diffeology by $ f $, i.e.,
		$ \{ P\in\mathrm{Param}(X)\mid f\circ P \mbox{ is a plot in } Y \} $
		is the same as the diffeology of $ X $.
		In particular, inductions are smooth.
	\end{definition}

	\begin{definition} 
		A smooth map $ f:X\rightarrow Y $ between diffeological spaces is a \textbf{local induction} if for every $ x $ in $ X $, there are D-open neighborhoods $ O\subseteq X $ of $ x $ such that $ f|_O:O\rightarrow Y $ is an injection and for every plot $ P:U\rightarrow f(O) $
		with $ P(0)=f(x) $,  $ f^{-1}\circ P $ is a plot in $ X $.
	\end{definition}

	\begin{definition}
		Let $ X $ be a diffeological space. A \textbf{diffeological subspace} of $ X $ is a subset $ X'\subseteq X $ equipped with the \textbf{subspace diffeology}, which is the set of all plots in $ X $ with values in $ X' $.
		In this situation, the inclusion map $ X'\hookrightarrow X $ is an induction.
	\end{definition}
	
	Every subspace $ X' $ of a diffeological space $ X $ has two topologies: 
	the D-topology given 
	by the subspace diffeology of $ X $ on $ X' $, 
	and the subspace topology inherited from the D-topology 
	of $ X $ on $ X' $. 
	Since the inclusion map is D-continuous, 
	any open subset of $ X' $ for the subspace topology is a D-open subset.
	\begin{definition}	
		A subspace $ X' $ of a diffeological space $ X $ is said \textbf{embedded} if its D-topology coincides with the subspace topology.
	\end{definition}
	
	\begin{definition}
		A map $f:X\rightarrow Y$ between diffeological spaces is a \textbf{subduction} if 
		the collection
\begin{center}
			$ \lbrace f\circ P\mid P $ is a plot in $ X\rbrace $
\end{center}
		is a covering generating family for $ Y $. Obviously, every subduction is a surjective smooth map.
	\end{definition}

	\begin{proposition}
				(\cite[\S 1.48]{PIZ}).
	A surjective smooth map $ f:X\rightarrow Y $ between diffeological spaces is a  subduction if and only if for any plot $ P:U\rightarrow Y $  and $ r_0\in U $, there exists at least one local lift plot $ L:V\rightarrow X $ defined on an open
	neighborhood $ V\subseteq U $ of $ r_0 $ such that $ f\circ L=P|_V $.
\end{proposition}

\begin{definition} 
	A smooth map $ f:X\rightarrow Y $ between  diffeological spaces is a \textbf{local subduction} if for every $ x_0\in X $ and for any plot $ P:U\rightarrow Y $ centered at $ f(x_0) $, there exists an open
	neighborhood $ V\subseteq U $ of $ 0 $ and a plot $ L:V\rightarrow X $ centered at $ x_0 $  such that $ f\circ L=P|_V $.
		Clearly, any local subduction is a  subduction.
\end{definition}

	\begin{proposition}\label{pro-dim-subd}
	(\cite[\S 1.82]{PIZ}).
	If	$f:X\rightarrow Y$ is a subduction, then $ \mathrm{dim}(X)\geq\mathrm{dim}(Y) $.
\end{proposition}

	\begin{proposition}\label{pro-subd-smd}
		(\cite[\S 1.51]{PIZ}).
		Let $\pi:X'\rightarrow X$ be a subduction. A map $ f:X\rightarrow Y $ is smooth if and only
		if $ f\circ\pi $ is smooth. Moreover, the map $ f $ is a subduction if and only if $ f\circ\pi $ is a subduction.
	\end{proposition}

	\begin{definition}
		Let $ X $ be a diffeological space.
		If $ R $ is an equivalence relation on $ X $ and $ q:X\rightarrow X/R $ is the quotient map, the diffeology generated by the family 
		$ \lbrace q\circ P\mid P $ is a plot in $ X\rbrace $ 
		is called the \textbf{quotient diffeology} on $ X/R $, and $ X/R $ is the \textbf{quotient space}.
		In this situation, the quotient map $ q:X\rightarrow X/R $ is a subduction.
	\end{definition}
 
	\begin{definition}
		Let $ \lbrace X_i\rbrace_{i\in J} $ be a family of diffeological spaces. 
		The \textbf{product diffeology} on
		$ X=\prod_{i\in J} X_i$ is given by the parametrizations $ P $ in $ X $ for which
		$ \pi_i\circ P$ is a plot in $X_i$ for all $ i\in J $, where
		$ \pi_i:X\rightarrow X_i$ is the natural projection.
		If $ J=\lbrace1,\ldots,n\rbrace $, then the plots in the product $X=X_1\times\dots\times X_n $ are
		$ n $-tuples $ (P_1,\ldots,P_n) $ where each $ P_i $ is a plot in $ X_i $.
	\end{definition}
	
	\begin{proposition}\label{lem-prd-top}
		(\cite[Lemma 4.1]{CSW2014}).
		Let $ X $ and $ Y $ be two diffeological spaces and $ D(X) $ is locally compact
		Hausdorff. Then $ D(X \times Y) $ and $ D(X) \times D(Y) $ are homeomorphic.
	\end{proposition}
	In particular, for a domain $ U $, $ D(U \times X) $ and $ U \times D(X) $ are homeomorphic.
	\begin{definition}
		Let $ \lbrace X_i\rbrace_{i\in J} $ be a family of diffeological spaces.
		For each $ i\in J $, there is a canonical
		injection $ X_i\rightarrow\bigsqcup_{i\in J} X_i $, which identifies $ X_i $ with its image in $ \bigsqcup_{i\in J} X_i $.
		The \textbf{disjoint union} or \textbf{sum diffeology} on $ X=\bigsqcup_{i\in J} X_i $ is given by the following
		property: A parametrization $ P:U\rightarrow X $ is a plot if there exists a partition $ \lbrace U_i\rbrace_{i\in J} $ of $ U $ such that $ P_i=P|_{U_i} $ is a plot in $ X_i $.
		In this situation, for each $ i\in J $, the canonical
		injection $ X_i\rightarrow X $ is an induction.
	\end{definition}
	\begin{proposition}
		Let $ \lbrace X_i\rbrace_{i\in J} $ be a family of diffeological spaces.  Then
		$ D(\bigsqcup_{i\in J} X_i) $ is homeomorphic to $ \bigsqcup_{i\in J} D(X_i) $.
	\end{proposition}
	\begin{proof}
		This result follows immediately from the fact that for any subset $ A\subseteq \bigsqcup_{i\in J} X_i $ and
		for every plot $ P:U\rightarrow \bigsqcup_{i\in J} X_i $, we have
		$ P^{-1}(A)=\bigsqcup_{i\in J}P^{-1}(A\cap X_i)  $.
	\end{proof}
	
	\begin{definition}
		Let $ X $ be a diffeological space and let $ \mathcal{C} $ be a parametrized cover of $ X $.
		The sum space $ \mathrm{Nebula}(\mathcal{C})=\bigsqcup_{P\in \mathcal{C}} \mathrm{dom}(P) $ is called the \textbf{nebula} of the parametrized cover $ \mathcal{C} $.
	\end{definition}
	\begin{proposition}
		(\cite[\S 1.76]{PIZ}). 
		Let $ X $ be a diffeological space and let $ \mathcal{C} $ be a parametrized cover of $ X $.
		Then $ \mathcal{C} $ is a covering generating family if and only if the evaluation map
	\begin{center}
			$ ev:\mathrm{Nebula}(\mathcal{C})\rightarrow X, ~~~(P,r)\mapsto P(r), $ 
	\end{center}
		is a subduction.
	\end{proposition}
	\subsubsection{The functional diffeology}
	\begin{definition}
		Let $X$ and $Y$ be diffeological spaces. 
		A parametrization $ Q:V\rightarrow \mathrm{C}^{\infty}(X,Y)$ is a plot for the \textbf{functional diffeology} on $ \mathrm{C}^{\infty}(X,Y) $ if for every plot $P:U\rightarrow X$ in $X$, the parametrization
		$ Q\circledcirc P:V\times U\rightarrow Y$
		given by
		$ (Q\circledcirc P)(s,r)= Q(s)\big{(}P(r)\big{)} $
		is a plot in $ Y $. 
	\end{definition} 
	For the functional diffeology, the natural map
	\begin{center}
		$ \mathrm{C}^{\infty}(X,\mathrm{C}^{\infty}(Y,Z)) \longrightarrow\mathrm{C}^{\infty}(X\times Y,Z) $
	\end{center}
	taking $ f \mapsto \tilde{f}$ with $\tilde{f}(x,y)=f(x)(y) $, is a diffeomorphism (see \cite[\S 1.60]{PIZ}). 
	Actually, the category of $ \mathsf{Diff} $ is Cartesian closed. 
	\begin{definition}
		Let  $ (X, \mathcal{D}) $ be a diffeological space. 
		A parametrization $ \rho:U\rightarrow\mathcal{D} $ is a plot in $ \mathcal{D} $ or a \textbf{smooth family of plots} in $ X $ if
		for every $ r_0\in U $ and $ s_0\in\mathrm{dom}(\rho(r_0)) $, there exist an open neighborhood $  V\subseteq U $ of $ r_0 $
		and an open neighborhood $  W $ of $ s_0 $ such that $ W \subseteq\mathrm{dom}(\rho(r)) $ for all $ r\in V $, and $ (r,s)\mapsto \rho(r)(s)$ defined on $ V\times W $ is a plot in $ X $.
	\end{definition}

\begin{notation}
			Let  $ (X, \mathcal{D}) $ be a diffeological space. 
	We denote the space  of centered 1-plots, i.e., defined on an open neighborhood of $ 0 $, by $ 1$-$\mathrm{Plots}_0(X)  $  together with the subspace diffeology inherited from $ \mathcal{D} $. 
\end{notation}

 
	\begin{definition}
		Any global $ 1 $-plot $ \mathbb{R}\rightarrow X $  is called a \textbf{path}, \textbf{curve},  or \textbf{homotopy} in $ X $, depending
		on context. 		
		The space of all paths in $ X $  equipped with the functional diffeology is denoted by 
		$ {\textsf{Paths}}(X) $.
		A path $ \gamma\in {\textsf{Paths}}(X) $ is called \textbf{stationary} if it is constant on some open neighborhood of
		$ (-\infty,0] $ and open neighborhood of $ [1,\infty) $.
		The subspace of the stationary paths is denoted by $ {\textsf{Paths}}_{st}(X) $.
	\end{definition}
	
	\begin{definition}
		A \textbf{diffeological group} is a group equipped with a \textbf{group diffeology} such that the multiplication and the inversion are
		smooth.
	\end{definition}
	
	\begin{example}\label{exa-diff}
		Let $ X $ be a diffeological space. Let $ \mathrm{Diff}(X) $ denote the group of diffeomorphisms on $ X $ with the composition operation.
		A parametrization $ P : U \rightarrow\mathrm{Diff}(X) $ is a plot for the \textbf{standard
			diffeology} of the group of diffeomorphisms if $ P $ itself and the parametrization $ r\mapsto P(r)^{-1} $ are both plots for functional
		diffeology (see \cite[\S 1.61]{PIZ}). Then  $ \mathrm{Diff}(X) $  is a diffeological group thanks to the smoothness of the composition operation \cite[\S 1.59]{PIZ}.
	\end{example}
	
	\begin{definition}\label{def-g-st}
		A \textbf{smooth action} of a diffeological group $ G $ on a diffeological space $ X $ is a smooth
		homomorphism $ G \rightarrow\mathrm{Diff}(X) $.
		A subspace  $ S\subseteq X $ is said to be \textbf{$ G $-stable} if it is connected and for every $ g\in G $, either $ g(S)=S $ or $ g(S)\cap S=\varnothing $.
	\end{definition}

	\subsubsection{Pullbacks of smooth maps}
	Pullbacks in the category $ \mathsf{Diff} $ of diffeological spaces exist.
	This construction can be concretely made as follows.
	Let $ f:X\rightarrow Y $ and $ g:Z\rightarrow Y $ be two smooth maps.
	The \textbf{pullback} of $ f $ by $ g $ denoted by $ g^*f:g^*X\rightarrow Z $ is a smooth map defined on the subspace
	\begin{center}
		$ g^*X=\lbrace(z,x)\in Z\times X\mid g(z)=f(x)\rbrace\subseteq Z\times X, $
	\end{center}
	taking $ (z,x) $ to $ z $. In fact, $ g^*f $ is the restriction of the projection $ \Pr_1:Z\times X\rightarrow Z $ to $ g^*X $.
	This 
	gives rise to a natural morphism $ (g_{\#},g) $
	\begin{displaymath}
		\xymatrix{
			g^*X \ar[r]^{g_{\#}}\ar[d]_{g^*f} & X\ar[d]^{f} \\
			Z \ar[r]^{g} & Y  }
	\end{displaymath}
	from $ g^*f $ to $ f $ in the category $ \mathrm{Mor}(\mathsf{Diff}) $ of morphisms of $ \mathsf{Diff} $ (see, e.g., \cite{M}), where $ g_{\#}=\Pr_2|_{g^*X} $ and $ \Pr_2:Z\times X\rightarrow X $ is the projection on the second factor.
	It is clear that for each $ z\in Z $, the fiber of $ g^*f $ over $ z $ is diffeomorphic to the fiber of $ f $ over $ g(z) $.
	Symmetrically, $ g_{\#} $ plays the role of the pullback of $ g $ by $ f $.
	
	One important property of pullbacks is transitivity.
	If $ h:W\rightarrow Z $ is another smooth map, then the pullback of $ g^*f $ by $ h $, i.e., $ h^*g^*f $, is equivalent to the pullback of $ f $ by $ g\circ h $, 
	i.e., $  (g\circ h)^*f $ in $ \mathrm{Mor}(\mathsf{Diff}) $, and that $ (g\circ h)_{\#}=g_{\#}\circ h_{\#} $. 
	In addition, the pullback of $ f:X\rightarrow Y $ by $ \mathrm{id}_Y:Y\rightarrow Y $ is equivalent to $ f $ itself in $ \mathrm{Mor}(\mathsf{Diff}) $, where $ \mathrm{id}_Y^*~X $ is nothing but the graph of $ f $.

	\begin{definition} 
		(\cite[\S 8.9]{PIZ}).
		A \textbf{diffeological fiber bundle} of fiber type $ T $ is a smooth surjective  map $ \pi:E\rightarrow X $ locally trivial along the plots in $ X $, that is,  the pullback of $ \pi $ by every plot in $ X $ is 
		locally trivial with fiber $ T $.  
	\end{definition}
	\begin{definition} 
				(\cite[\S 8.22]{PIZ}).
		A \textbf{diffeological covering map} is a diffeological fiber bundle with discrete fibers. 
		A \textbf{diffeological covering space} is the total space of a diffeological covering map. 
	\end{definition}
	\subsection{Internal tangent spaces  and tangent bundles}
	We here recall the internal tangent spaces  and tangent bundle of a diffeological space $ X $ (see \cite{CW,Hec,HM-V} for more details).
	
	Fix any $ x\in X$. The category of germs of plots centered at $ x $, denoted by $ \mathcal{G}\mathsf{Plots}_x(X) $,
	has plots centered at $ x $ for objects and 
	a morphism $ Q\stackrel{\mathcal{G}_x(F)}{\longrightarrow} P $ 
	between two plots $ P:U\rightarrow X $ and $ Q:V\rightarrow X $ centered at  $ x $ is
	the germ class of a smooth map $ F:W \rightarrow U $, defined on an open neighborhood $ W\subseteq V $ of $ 0 $, taking $ 0 $ to $ 0 $ such that $ Q|_W=P\circ F $. Notice that
	two such maps have the same germ if they agree on an open neighborhood of $ 0 $ in $ V $.
	The \textbf{internal tangent space} $ T_xX $ of $ X $ at $ x $ is the colimit of the functor
	$ T_0:\mathcal{G}\mathsf{Plots}_x(X)\rightarrow \textsf{Vect} $
	from the category $ \mathcal{G}\mathsf{Plots}_x(X) $ to that of vector spaces and linear maps,
	given by
	\begin{center}
		$ Q\stackrel{\mathcal{G}_x(F)}{\longrightarrow} P\qquad\longmapsto\qquad 
		dF_0:T_0V\longrightarrow T_0U $,
	\end{center}
	where $ dF_0 $ is the usual differential of $ F $ at $ 0 $.
	Thus, we have a cocone as the following:
	$$
	\xymatrix{
		& T_x X  & \\
		T_0 V  \ar[ur]^{dQ_0} \ar[rr]_{dF_0} & & T_0U  \ar[ul]_{dP_0}
	}
	$$
	in which  $ dP_0 $ denotes the linear map given by the colimit for a plot $ P $, so that
	$ dP_0\circ dF_0=dQ_0 $, or equivalently 
	$ dP_0\circ dF_0=d(P\circ F)_0 $.
	If  two $ n $-plots $ P $ and $ Q $ centered at $ x $
	have the same germ at $ 0 $, then $ dP_0=dQ_0 $.

\begin{definition}
	Let $ f:X\rightarrow Y  $ be a smooth map and $ x\in X$. The universal property of colimits induces a unique linear map
	$ df_x:T_xX\rightarrow T_{f(x)}Y  $, called the \textbf{internal tangent map} of $ f $ at $ x $, with the property that
	$ df_x\circ  dP_0= d(f\circ P)_0 $,
    for all plot $ P $ in $ X $  centered at $ x\in X $.
\end{definition}
	As mentioned in \cite[p. 11]{CW}, the above construction  leads to a functor from the category of pointed diffeological spaces to that of vector spaces.
Therefore, the following properties hold.
\begin{proposition}
Let $ f:X\rightarrow Y  $ and $ g:Y\rightarrow Z $ be smooth maps, and  $ x\in X $. 
 	\begin{enumerate}
 	\item[ (1) ] 	
 	 $ d(\mathrm{id}_X)_x=\mathrm{id}_{T_xX} $.
 	\item[ (2) ] 
 	 $ d(g\circ f)_x=dg_{f(x)}\circ df_x $.
	\item[ (3) ] 
      If $ f $ is a diffeomorphism, then $ df_x $ is an isomorphism, and $ (df_x)^{-1}=d(f^{-1})_{f(x)} $.
 \end{enumerate} 
\end{proposition}

\begin{proposition}\label{pro-cons}
	If $ f:X\rightarrow Y  $ is a locally constant map and $ x\in X $, then 
	$ df_x=0 $.
\end{proposition}	
\begin{proof}
	First notice that if a plot $ P:U \rightarrow X $ centered at $ x\in X $ is constant,
	then $ P\circ c_x=P $, where $ c_x:U\rightarrow U $ is the constant map with the value $ 0 $.
	Therefore, $ dP_0=d(P\circ c_x)_0=dP_0\circ  dc_x=0 $.
	Now for an arbitrary plot $ P $  in $ X $ centered at $ x\in X $, the composition $ f\circ P $ and a constant plot with the value $ f(x) $	have the same germ at $ 0 $.
	Hence
	$ df_x\circ  dP_0= d(f\circ P)_0=0 $ and the result follows by the uniqueness of internal tangent map with this property.
\end{proof}

\begin{definition}
	The \textbf{internal tangent bundle} $ TX $ of a diffeological space $ X $, 
	as a set, is the disjoint union $ \bigsqcup_{x\in X}T_xX $.
	There are two diffeologies on $ TX $:
	\begin{enumerate}
		\item[$   $] 
		\textbf{Hector's diffeology:} This diffeology is generated by the maps
		$ dP:TU\rightarrow TX $, corresponding to all plots $ P:U\rightarrow X $ in $ X $,
		where $ TU $ has the standard diffeology (as a domain) and
		$ dP(r,u)=(P(r),dP_r(u)) $,
		for each $ (r,u)\in TU $, where $ dP_r:T_rU\rightarrow T_{P(r)}X  $ is the  internal tangent map of $ P $ at $ r\in U $.
		The internal tangent bundle of $ X $ endowed with Hector's diffeology is
		denoted by $ T^HX $.\\
		\item[$   $] 
		\textbf{dvs diffeology:}
		This is the smallest diffeology  on $ TX $,  containing
		Hector's  diffeology, that turns $ TX $ into a diffeological vector space over $ X $.
		The internal tangent bundle of $ X $ endowed with the dvs diffeology is
		denoted by $ T^{dvs}X $.
	\end{enumerate} 
\end{definition}
As mentioned in \cite[Example 3.6]{AD}, the internal tangent bundle of a diffeological space $ X $ constitutes a cosheaf on the site of plots.
 \begin{definition}
 	A \textbf{vector field} on a diffeological space $ X $ is any smooth map $ \Lambda:X\rightarrow T^HX $ with $ \pi_X\circ\Lambda=\mathrm{id}_X $, where $ \pi_X:T^HX\rightarrow X $ is the canonical projection.
 	Denote the space of vector fields on $ X $ by  
 	$ \mathfrak{X}(X) $, equipped with the subspace diffeology inherited from $ \mathrm{C}^{\infty}(X,T^HX) $.
 \end{definition}

\begin{definition}
	If $ f:X\rightarrow Y  $ is a smooth map, the map
	$ df:T X \rightarrow T Y $ defined by
	$ df(x,u)=(f(x),df_x(u)) $,
	is called the \textbf{global internal tangent map} of $ f $ at $ x $.
\end{definition}
The following properties for global internal tangent maps are immediate:
\begin{proposition}\label{pro-gtm-sm}
	Let $ f:X\rightarrow Y  $ and $ g:Y\rightarrow Z $ be smooth maps. 
	\begin{enumerate}
		\item[ (a) ] 	
		$ d(\mathrm{id}_X)=\mathrm{id}_{TX} $.
		\item[ (b) ] 
		$ d(g\circ f)=dg\circ df $.
		\item[ (c) ] 
		If $ f $ is a diffeomorphism, then $ df $ is a diffeomorphism, and $ (df)^{-1}=d(f^{-1}) $.
		\item[ (d) ] 
		$ df:T^HX \rightarrow T^HY $ is a morphism of smooth projections, i.e., a smooth map making the following diagram commute:
		\begin{displaymath}
			\xymatrix{
				T^HX\ar[d]_{\pi_X}\ar[r]^{df}	&  T^HY\ar[d]^{\pi_{Y}}   \\
				X \ar[r]_{f} & Y  \\
			}
		\end{displaymath}
	\end{enumerate} 
\end{proposition}

	\section{Submersions in diffeology}\label{S3}
	
In this section, we introduce and investigate diffeological submersions, compare them with submersions of manifolds.
Let us begin with a few related concepts close to them.
	\begin{definition}
		A smooth map $ f:X\rightarrow Y $ between diffeological spaces is a \textbf{strong subduction} if there is a smooth map
		$ \sigma:Y\rightarrow X $, called a smooth \textbf{section} of $ f $, with the property that  $ f\circ\sigma=\mathrm{id}_Y $.
	\end{definition}
	\begin{example} 
		Let $ \lbrace X_i\rbrace_{i\in J} $ be a family of nonempty diffeological spaces. 
		The canonical projection $ \prod_{i\in J} X_i\rightarrow X_i $  is a strong subduction, for each $ i\in J $.	
	\end{example}
	\begin{example}
	Let $ X $ be a diffeological space. 
	The canonical projection $ \pi_X:TX \rightarrow X $ is a strong subduction, when $ TX $ is equipped with either Hector's diffeology or the dvs diffeology.
\end{example}

\begin{example}\label{exa-sub-sum}
	The surjective smooth map $ f:\mathbb{R}\setminus\{0\}\rightarrow (0,+\infty) $
	given by
\begin{align*}
		f(x)=\left \{
\begin{array}{lr}
			x,\quad &  x>0\\
			1, \quad &  x<0\\
\end{array} \right.
\end{align*}
	is a strong subduction but not a submersion. 
\end{example}

\begin{definition}
	A smooth map $ f:X\rightarrow Y $ between diffeological spaces is a \textbf{global subduction} if for any plot $ P:U\rightarrow Y $, there exists at least one global lift plot $ L:U\rightarrow X $ such that $ f\circ L=P $.
\end{definition}
\begin{example}
	Any linear subduction between diffeological vector spaces  is a global subduction by \cite[Corollary 6.5]{WU}.
\end{example}
	\begin{definition}
	A (not necessarily surjective) smooth map $ f:X\rightarrow Y $ between diffeological spaces is a \textbf{weak subduction} if for any plot $ P:U\rightarrow Y $  and $ r_0\in U $ with $ P(r_0)\in f(X) $, there exists at least one local lift plot $ L:V\rightarrow X $ defined on an open
	neighborhood $ V\subseteq U $ of $ r_0 $ so that $ f\circ L=P|_V $.
	\end{definition}
		
	\begin{proposition} 
		If $ f:X\rightarrow Y $ is a weak subduction, then
		$ f(X) $ is a D-open subset of $ Y $ and
		$ \mathrm{dim}(X)\geq\mathrm{dim}(f(X)) $, where $ f(X)\subseteq Y $ is equipped with the subspace diffeology.
	\end{proposition}
	\begin{proof}
		It is easy to see that $ f(X) $ is a D-open subset of $ Y $.
		Because $ f $ is a subduction onto its image as a subspace of $ Y $, the result is obtained by Proposition \ref{pro-dim-subd}.
	\end{proof}
Clearly, we have the following inclusions:
\begin{center}
		$\{$ Strong subductions $\}\subseteq\{$ Global subductions $\}\subseteq\{$ Subductions $\}\subseteq\{$ Weak subductions $\}$ 
\end{center}

	\begin{definition}\label{def-sub}
		We call a smooth map $ f:X\rightarrow Y $ between diffeological spaces a \textbf{submersion} if for each $ x_0 $ in $ X $, there exists a smooth \textbf{local section} $ \sigma:O\rightarrow X $ of $ f $
		passing through $ x_0 $ defined on a D-open subset $ O\subseteq Y $ such that $ f\circ\sigma(y)=y $ for all $ y\in O $, that is, $ x_0\in \sigma(O) $ and
		the following diagram commutes;
		\begin{displaymath}
			\xymatrix{
				& X \ar[d]^{f} \\
				O\ar[r]_{\iota} \ar@{-->}[ur]^{\sigma} 	& Y }
		\end{displaymath}
		where $ \iota:O\hookrightarrow Y $ is the canonical inclusion.
		In this situation, $ \sigma(O)\subseteq f^{-1}(O) $, $ f(x_0)\in O $ and $ \sigma $ takes $ f(x_0) $ to $ x_0 $.
	\end{definition}
	\begin{remark}\label{rem-sub}
		As the pullback $ \iota^*f:f^{-1}(O)\rightarrow O $ is a strong subduction in Definition \ref{def-sub}, we can say that if $ f:X\rightarrow Y $ is a submersion, then there exists a D-open cover $ \{O_i\}_{i\in J} $ of $ f(X) $ in $ Y $ such that 
		$ \iota_i^*f:f^{-1}(O_i)\rightarrow O_i $ is a strong subduction for all $ i\in J $.
		The empty map $ \varnothing\rightarrow X $ is a submersion by definition.
	\end{remark}  
	
	\begin{proposition}\label{pro-comp-pb-sub}
		The composition of two submersions is again a submersion. The pullback of a submersion by any smooth map remains a submersion.
	\end{proposition} 
	\begin{proof}
		The proof is straightforward.
	\end{proof}
	
	\begin{proposition} 
		A map between manifolds is a submersion of diffeological spaces if and only if it is a submersion of manifolds.
	\end{proposition}
	\begin{proof}
		It is sufficient to apply the local section theorem \cite[Theorem 4.26]{Lee}.
	\end{proof} 

	\begin{definition}
		A smooth map $ f:X\rightarrow Y $  is said to be a \textbf{diffeological submersion} if
		the pullback  of $ f $ by every plot in $ Y $ is a submersion of diffeological spaces.
		
	\end{definition}

 \begin{proposition}
 	A diffeological submersion is a D-open map.
 \end{proposition}
	\begin{proof}
	The proof is entirely analogous to that of \cite[\S 2.18]{PIZ}.
	\end{proof}
	A smooth map $ f:X\rightarrow Y $ is a diffeological submersion if and only if there exists a covering generating family  $ \mathcal{C} $ of  $ Y $ such that
	for any plot $ P $ which belongs to $ \mathcal{C} $, 
	the pullback $ P^{*}f $ is a submersion.

	\begin{proposition}\label{pro-cov-dsub} 
		If $ f:X\rightarrow Y $ is a  diffeological submersion,
		then there exists a covering generating family  $ \mathcal{C} $ of the subspace $ f(X)\subseteq Y $ such that
		for any plot $ Q $ which belongs to $ \mathcal{C} $, 
		the pullback $ Q^{*}f $ is a strong subduction.
	\end{proposition}
	\begin{proof}
		Take  the collection $ \mathcal{C} $ of all of the plots $ Q $ in $ Y $ for which the pullback  of $ f $ by  $ Q $ is a strong subduction. All constant plots with values in $ f(X) $ belong to this collection, so $ \mathcal{C} $ is a parametrized cover of $ f(X) $.
		Let $ P:U\rightarrow f(X) $ be an arbitrary plot.
		The pullback $ P^{*}f:P^{*}X\rightarrow U $ is a surjective submersion 
		and by Remark \ref{rem-sub}, there exists an open cover $ \{U_i\}_{i\in J} $ of $ U $ such that 
		$ (P|_{U_i})^*f $ is a strong subduction for all $ i\in J $.
		Therefore, $ P $ is as the supremum of a family of plots belonging to $ \mathcal{C} $, and hence $ \mathcal{C} $ is a covering generating family for $ f(X) $.
	\end{proof} 
	
	\begin{proposition} 
		A smooth map $ f:X\rightarrow Y $ between  diffeological spaces is a diffeological submersion if and only if for given $ x_0\in X $, for any plot $ P:U\rightarrow Y $, and for $ r_0\in U $ with $ P(r_0)=f(x_0)  $, there exists at least one local lift plot $ L:V\rightarrow X $ defined on an open
		neighborhood $ V\subseteq U $ of $ r_0 $ such that $ f\circ L=P|_V $ and $ L(r_0)=x_0 $.
	\end{proposition}
	\begin{proof}
		Assume that $ f:X\rightarrow Y $ is a diffeological submersion and let $ x_0\in X $, $ P:U\rightarrow Y $ be a plot, and $ r_0\in U $ with $ P(r_0)=f(x_0)  $. Because the pullback
		$ P^*f:P^*X \rightarrow U $
		is a submersion and $ (r_0,x_0)\in P^*X $, there is 
		a smooth local section $ \sigma:V\rightarrow P^*X $ of $ P^*f $
		defined on an open neighborhood $ V \subseteq U $ of $ r_0 $  with $ \sigma(r_0)=(r_0,x_0) $.
		Then the composition $ L=P_{\#}\circ \sigma $ is a plot in $ X $ with $ f\circ L=P|_V $ and $ L(r_0)=x_0 $.
		\begin{displaymath}
			\xymatrix{
				&	P^*X  \ar[r]^{P_{\#}}\ar[d]_{P^*f } & X \ar[d]^{f} \\
				V\ar@{-->}[ru]^{ \sigma} \ar[r]^{\iota}	&	U \ar[r]^{P} & Y  }
		\end{displaymath} 
		
		Conversely, we show that for any plot $ P:U\rightarrow Y $ the pullback $ P^*f $ is a submersion. Let $ (r_0,x_0)\in P^{*}X $ or $ P(r_0)=f(x_0)  $. By hypothesis, there exist an open neighborhood $ V $ of $ r_0 $ and a local lift plot $ L:V\rightarrow X $ such that $ f\circ L=P|_V $ and $ L(r_0)=x_0 $.
		Therefore, one can construct a smooth local section $ \sigma:V\rightarrow P^*X $ of  $ P^*f  $ by $ \sigma(r)=(r,L(r)) $
		satisfying $ \sigma(r_0)=(r_0,x_0) $.
	\end{proof}

	\begin{corollary}\label{coro-dim-dsub} 
		A diffeological submersion is a weak subduction. Thus,	for a diffeological submersion $ f:X\rightarrow Y $, we get		$ \mathrm{dim}(X)\geq\mathrm{dim}(f(X)) $, where $ f(X)\subseteq Y $ is equipped with the subspace diffeology.
	\end{corollary}

	\begin{corollary}
	Surjective diffeological submersions are the same as local subductions.
\end{corollary}

It is clear that every diffeomorphism is a surjective diffeological submersion or local subduction,
and every surjective diffeological submersion is a hyper-diffeomorphism (see \cite[Definition 4.8]{ADD}).
Also, an injective diffeological submersion is an induction and so a diffeomorphism onto its image.

	\begin{proposition}
		If $ f:X\rightarrow Y $ is a diffeological submersion, then the internal tangent map $ df_x:T_xX\rightarrow T_{f(x)}Y $ is an epimorphism, at each point of $ x\in X $.
	\end{proposition}
	\begin{proof}
	We check that $ df_x:T_{x}X\rightarrow T_{f(x)}Y $ is surjective.
	Let $ u=\sum_{i=1}^{k} d(P_i)_0(u_i)$ be an element of $ T_{f(x)}Y $, where $ P_i $ are plots in $ Y $ centered at $ f(x) $, for $ i=1,\cdots,k $,
and $ u_i\in T_0\mathrm{dom}(P_i) $.
Since $ f $ is a diffeological submersion and $ f(x)=P_i(0) $, one can find a
local lift plot $ L_i $ of $ P_i $ along $ f $ with $ L_i(0)=x $, for each $ i=1,\cdots,k $.
Then $ df_x(\sum_{i=1}^{k} d(L_i)_0(u_i))=\sum_{i=1}^{k} d(f\circ L_i)_0(u_i)=\sum_{i=1}^{k} d(P_i)_0(u_i)=u $.
Hence $ df_x $ is an epimorphism.
	\end{proof}

	\begin{example}
		Suppose that $ G $ is a diffeological group and  $ X $ be a diffeological space, and
		let $ \alpha:G \rightarrow\mathrm{Diff}(X) $ be a smooth action of $ G $ on $ X $.
			\begin{enumerate}
			\item[ (a) ] 	
	The adjoint  map $ \overline{\alpha}:G\times X\rightarrow X, \overline{\alpha}(g,x)=\alpha(g)(x) $ is a strong subduction.
	In particular, 	the operation of a diffeological group is a strong subduction.
			\item[ (b) ] 
		 The quotient map $ \pi:X\rightarrow X/G $ is a surjective diffeological submersion or a local subduction. 	In fact, let $ P:U\rightarrow X/G $ be any plot, $ x_0\in X $ and $ r_0\in U $ such that $ P(r_0)=\pi(x_0)  $.
		 Because $ \pi $ is a subduction, there exist an open neighborhood $ V $ of $ r_0 $ and a plot $ Q:V\rightarrow X $ such that $ \pi\circ Q=P|_V $.
		 In particular, $ \pi (Q(r_0))=\pi(x_0) $
		 and there is an element $ g $ of $ G $ for which $ \alpha(g)(Q~~(r_0))=x_0 $.
		 Now it is sufficient to set $ L=\alpha(g)\circ Q $, which is a local lift plot of $ P|_V $ along $ \pi $ with $ L(r_0)=x_0 $.
		\end{enumerate} 
	\end{example}

  	The following is an example of a  diffeological submersion, which may not be local subduction.
  \begin{example} 
  	Let $ X $ be a (not necessarily connected) diffeological space.
  	The map $ ends:\mathsf{Paths}_{st}(X)\rightarrow X\times X $ taking any stationary path $ \gamma $ in $ X $ to the pair $ (\gamma(0),\gamma(1))  $ is a diffeological submersion.
  	To see this, assume that  $ \gamma $ is a stationary path in $ X $, $ (P,P'):U\rightarrow X\times X $ is a plot, and let $ r_0\in U $ be such that $ (P(r_0),P'(r_0))=(\gamma(0),\gamma(1))  $.
  	Then there exists some $ \epsilon>0 $ such that $ \gamma(t)=P(r_0) $ for all $ t<\epsilon $, and $ \gamma(t)=P'(r_0) $ for all $ 1-\epsilon<t $.
  	Let $ B_{r_0} $ be an open ball in $ U $ centered at $ r_0 $. Then $ P(B_{r_0})$ and $P'(B_{r_0}) $ are both connected and lie in the same connected component of $ X $. Construct  the plot $ \phi $ in $ \mathsf{Paths}_{st}(X) $ from $ B_{r_0} $ via
  	\begin{align*}
  		\phi(r)(t)=\left \{
  		\begin{array}{lr}
  			P\big(\lambda(\frac{t}{\epsilon})r_0~-\lambda(\frac{t}{\epsilon})r+~r\big),\quad &  t < \epsilon\\
  			\gamma(t), \quad &  \epsilon\leq t \leq 1-\epsilon\\
  			P'\big(\lambda(\frac{t+\epsilon-1}{\epsilon})r~-\lambda(\frac{t+\epsilon-1}{\epsilon})r_0+~r_0\big),\quad &  1-\epsilon < t\\
  		\end{array} \right.
  	\end{align*}
  	for all $ r\in B_{r_0}  $, $ t\in\mathbb{R} $,
  	where
  	$ \lambda $
  	is the smashing function
  	given in \cite[\S 5.5]{PIZ}.
  	Thus,
  	$ \mathsf{ends}\circ \phi=(P,P')|_{B_{r_0}} $ and $ \phi(r_0)=\gamma $.
  \end{example}

	It is easy to see that diffeological submersions are closed under their compositions and pulling back by smooth maps.
	By Proposition \ref{pro-comp-pb-sub}, any submersion is a diffeological submersion.
	However, a diffeological submersion is not necessarily a  submersion.
	\begin{example}
		The diffeological principal bundle  $  \mathbb{T}^2 \rightarrow \mathbb{T}_{\alpha},~~~ \alpha\notin\mathbb{Q} $, is obviously a surjective diffeological submersion (see \cite[\S 8.15]{PIZ}). But it  is not a submersion;
		for otherwise,  it has a smooth global section and hence it must be trivial by \cite[\S 8.12, Note 2]{PIZ}.
		This follows that $ \mathbb{R}\times\mathbb{T}_{\alpha}\cong \mathbb{T}^2  $ is D-compact and so $ \mathbb{R} $ is compact, a contradiction!
	\end{example}
	
	On manifolds, however, the situation is different.
	
	\begin{proposition} 
		If $ f:X\rightarrow M $ is a diffeological submersion into a manifold $ M $, then it is a submersion.
	\end{proposition}
	\begin{proof}
		Let $ x_0\in X $ and $ \varphi:U\rightarrow M  $ be a chart\footnote{By a chart, we mean a diffeomorphism from a domain onto an open subset of a manifold.} around $ f(x_0) $, i.e., $ f(x_0)=\varphi(r_0) $ for some  $ r_0\in U $.
		So we obtain a local lift plot $ L:V\rightarrow X $ with $ r_0\in V\subseteq U $ such that $ f\circ L=\varphi|_V $ and $ L(r_0)=x_0 $.
		Obviously, $ \sigma:=L\circ(\varphi|_V)^{-1} $ is a smooth local section of $ f $ defined on the open subset
		$ \varphi(V)\subseteq M $
		with $ \sigma(\varphi(r_0))=x_0 $, and hence  $ f $ is a submersion. 
	\end{proof}
	
	\begin{corollary}\label{cor-dsub-mfd}
		A map between manifolds is a diffeological submersion if and only if it is a submersion of manifolds.
	\end{corollary}
	
	\begin{corollary}\label{cor-dsub-mfd2} 
		A smooth map $ f:X\rightarrow Y $ between  diffeological spaces is a diffeological submersion if and only if for given $ x\in X $, any smooth map  $ g:M\rightarrow Y $ from a manifold $ M $, and $ m\in M $ with $ g(m)=f(x)  $, there exists at least one smooth local lift  $ h:O\rightarrow X $ defined on a D-open
		neighborhood $ O\subseteq M $ of $ m $ such that $ f\circ h=g|_O $ and $ h(m)=x $.
	\end{corollary}

	\section{Immersions in diffeology}\label{S4}
	
In this section, we will  introduce diffeological immersions, explore their key properties, and show that they behave like immersions of manifolds. Accordingly, we  consider  notions for  embeddings of diffeological spaces. But first of all, we need some auxiliary concepts.
	\subsection{Diffeological injections}
	
	\begin{definition}
		A  (not necessarily injective) smooth map $ f:X\rightarrow Y $  is a \textbf{diffeological injection}
		if the following implication holds:
		
		\begin{enumerate}
			\item[$ \textbf{Inj.} $] 
			For any two $ n $-plots $ P$ and $Q $ in $ X $, and every $ r_0\in \mathrm{dom}(P)\cap\mathrm{dom}(Q) $, whenever $ P(r_0)=Q(r_0) $ and there exists an open neighborhood $ U $ of $ r_0 $ in $\mathrm{dom}(P)\cap\mathrm{dom}(Q) $ such that
			$ f\circ P|_U=f\circ Q|_U $,
			then $ P|_V=Q|_V $ for some open neighborhood $ V $ of $ r_0 $ in $U $.
			
		\end{enumerate}
	\end{definition}

	A diffeological injection need not be injective, although any smooth injective map, especially an induction, is a diffeological injection.
	Obviously, any map defined on a discrete space is a  diffeological injection.
	The pullback of a diffeological injection by a diffeological injection remains a diffeological injection.
	Also, the composition of two diffeological injections is again a diffeological injection.
	
	\begin{proposition}\label{pro-dinj-mfd}
		A  smooth map $ f:X\rightarrow Y $ is a diffeological injection if and only if for given $ x\in X $, any smooth map  $ g:M\rightarrow Y $ from a manifold $ M $, and $ m\in M $ with $ g(m)=f(x)  $, there exists at most one  smooth local lift  $ h:O\rightarrow X $ defined on a D-open
		neighborhood $ O\subseteq M $ of $ m $ such that $ f\circ h=g|_O $ and $ h(m)=x $.
	\end{proposition}
	\begin{proof}
		The proof is straightforward.	
	\end{proof}

	\begin{corollary}\label{cor-inj} 
	Suppose that $ M $ is a manifold and $ f:M\rightarrow X $ is a diffeological injection.
	If $ m_0\in M $ and $ g:O\rightarrow M $ is a smooth map defined on a D-open
	neighborhood $ O\subseteq M $ of $ m_0 $ such that $ g(m_0)=m_0 $ and $ f\circ g(m)=f(m) $ for all $ m\in O $,
	then there exists a D-open
	neighborhood $ O'\subseteq O $ of $ m_0 $ such that $ g(m)=m $ for all $ m\in O' $.
\end{corollary}

	\begin{proposition}
		If $ f:X\rightarrow Y $ is a diffeological injection, then every fiber  of  $ f $ is a discrete subspace of $ X $.
	\end{proposition}
	\begin{proof}
		Assume that $ y\in Y $ and the fiber $  X_y $ over $ y $ is nonempty, $ X_y=f^{-1}(y)\neq\varnothing $. Let $ P:U\rightarrow X $ be a plot with values in the fiber $ X_y $. Fix any $ r\in U $, set $ x:=P(r) $. Take the constant plot $ \widetilde{x} $ with the value $ x $ defined on $ U $.
		Then $ P(r)=\widetilde{x}(r) $ and $ f\circ P=f\circ \widetilde{x} $, which implies that
		$ P|_V=\widetilde{x}|_V $ on some open neighborhood $ V $ of $ r $ in $U $.
		Thus, $ P $ is locally constant and $ X_y $ is a discrete subspace.
	\end{proof}
	\begin{example}\label{exa-inj-dis}
		The quotient map  $ q:\mathbb{R}\rightarrow \mathbb{R}/O(1),~~~ x\mapsto\{-x,x\} $ of the natural action of the group $ O(1) $ on $ \mathbb{R} $ has discrete fibers by \cite[Exercise 8]{PIZ}, but it is not a diffeological injection. Because we have $ \mathrm{id}_{\mathbb{R}}(0)=-\mathrm{id}_{\mathbb{R}}(0) $ and $ q\circ \mathrm{id}_{\mathbb{R}}=q\circ (-\mathrm{id}_{\mathbb{R}}) $, but $ \mathrm{id}_{\mathbb{R}} $ and $ -\mathrm{id}_{\mathbb{R}} $ are not equal on any open neighborhood of $ 0 $.
	\end{example}

	\begin{proposition}\label{pro-inj-lift}
		Suppose that $ f:X\rightarrow Y $ is a diffeological injection and $ P:U\rightarrow Y $ is any plot in $ Y $. If $ L:U\rightarrow X $ is a lift plot of $ P $ along $ f $, then $\mathsf{Graph}(L)$ is a D-open subset of $ P^*X $.
	\end{proposition}
	\begin{proof}
		Assume that $ L:U\rightarrow X $ is a lift plot of $ P $ along $ f $, i.e., $ f\circ L=P $.
		Let $ (F,Q):V\rightarrow P^*X $ be any plot in $ P^*X $ so that $ P\circ F=f\circ Q $ and $ f\circ L\circ F=f\circ Q $. For $ r_0\in(F,Q)^{-1}(\mathsf{Graph}(L))  $, we get $ (F(r_0),Q(r_0))\in\mathsf{Graph}(L) $ or  $ L\circ F(r_0)= Q(r_0) $.
		As $ f $ is a diffeological injection, the equality $ L\circ F|_{W}=Q|_{W} $ is obtained on some open neighborhood $ W\subseteq V $ of $ r_0 $.
		Obviously, $  r_0\in W\subseteq (F,Q)^{-1}(\mathsf{Graph}(L)) $, and hence $ \mathsf{Graph}(L) $ is a D-open subset of $ P^*X $.
	\end{proof}
	
	\subsection{Diffeological immersions}
	
	\begin{definition}
		Let $ X\subseteq X' $ be a subspace of a diffeological space $ X' $.
		A smooth map $ f:X\rightarrow Y $ between diffeological spaces is a \textbf{strong induction} if there is a smooth map
		$ \rho:Y\rightarrow X' $ such that $ \rho\circ f(x)=x $ for all $ x\in X $.
	\end{definition}

	Any strong induction is an induction and a diffeomorphism onto its image.
	
	\begin{example} 
		Let $ \lbrace X_i\rbrace_{i\in J} $ be a family of nonempty diffeological spaces. Let $ A_i $ be a subspace of $ X_i $. The injection $ A_i\rightarrow\prod_{i\in J} X_i $ is a strong inductions.
	\end{example}

	\begin{definition}\label{def-imm}
		A smooth map $ f:X\rightarrow Y $ between diffeological spaces is an \textbf{immersion} if for each $ x_0 $ in $ X $, there exist a D-open neighborhood $ O\subseteq X $ of the point $ x_0 $, a D-open neighborhood $ O'\subseteq Y $ of the set $ f(O) $, and a smooth map $ \varrho:O'\rightarrow X $ such that 
		the following diagram commutes;
	\begin{displaymath}
		\xymatrix{
			  & \ar@{-->}[ld]_{\varrho} O'\\
			X & O\ar[l]_{\iota}\ar[u]_{f|_O}  \\
		}
	\end{displaymath} 
		where $ \iota:O\hookrightarrow X $ indicates the canonical inclusion.
	\end{definition}
	An immersion of diffeological spaces is, locally, a strong induction into a D-open subspace.
	In particular, it is a local induction.

	\begin{proposition} 
		A map  $ f:M\rightarrow N $ between manifolds is an immersion of diffeological spaces if and only if it is an immersion  of manifolds.
	\end{proposition}
	\begin{proof}
		First assume that $ f:M\rightarrow N $ be an immersion of diffeological spaces. For each $ m \in M $, by definition, there exist a D-open neighborhood $ O\subseteq M $ of $ m $, a D-open neighborhood $ O'\subseteq N $ of $ f(O) $ and a smooth map $ \varrho:O'\rightarrow M $ such that $ \rho\circ f|_{O}=\imath  $, where $ \iota:O\hookrightarrow M $ is the canonical inclusion. Thus, $ d\rho_{f(m)}\circ df_m=d\iota_m  $ is an isomorphism by \cite[Proposition 3.9]{Lee}
		and consequently, $ df_m $ is injective.
		Hence $ f $ is an immersion of manifolds.
		The
		converse follows from the rank theorem.
	\end{proof} 
	
	\begin{definition}\label{def-dimm}	 
		A smooth map $ f:X\rightarrow Y $  is a \textbf{diffeological immersion} if
		for any plot $ P:U\rightarrow Y $ in $ Y $,
		the pullback of $ f $ by $ P $ is an immersion, in the sense that for each $ (r_0,x_0) $ in $ P^*X $, there exist a D-open neighborhood $ O $ of $ (r_0,x_0) $ in $ P^*X $, an open neighborhood $ V\subseteq U $ of $ P^*f(O) $ and a smooth map $ \varrho:V\rightarrow U\times X $ such that $ \varrho\circ P^*f(r,x)=(r,x) $ for all $ (r,x)\in O $.
		Notice that 
		we do not require that $ \rho $ be a map into $ P^*X $.
		\begin{displaymath}
		\xymatrix{
			U\times X\ar@/^0.4cm/[drr]^{\Pr_2}& & \\
		&	P^*X \ar[ul]\ar[r]^{P_{\#}}\ar[d]^{P^*f} & X\ar[d]^{f} \\
		&	V\ar@/^0.6cm/@{-->}[uul]^{\varrho}\ar[r] \subseteq U \ar[r]^{P} & Y  }
	\end{displaymath}
	\end{definition}
	
	\begin{remark}\label{rem-dimm}	
		The condition of Definition \ref{def-dimm} implies that	$ O\subseteq\mathsf{Graph}(\Pr_2\circ\rho) $. Thus, \textit{if} $ L $ is any local lift plot of $ P $ along $ f $ and $ r_0\in \mathrm{dom}(L) $ with $ L(r_0)=x_0 $, then we can find an open neighborhood $ W\subseteq \mathrm{dom}(L) $ of $ r_0 $ such that 
		$ L|_W=\Pr_2\circ\rho~~|_W $.
	\end{remark}

\begin{example}(\cite{AM}).
	Let $ X $ be a diffeological space.
	For the weak power set diffeology on  the power set $ \mathfrak{P}(X) $, the map $ \imath:X\rightarrow\mathfrak{P}(X) $ defined by $ \imath(x)=\{x\} $ is a diffeological immersion.
\end{example}

	\begin{proposition}\label{pro-imm-inj}
		Any diffeological immersion is a diffeological injection.
	\end{proposition}
	\begin{proof}
		Suppose that $ f:X\rightarrow Y $ is a diffeological immersion.
		Let $ P$ and $Q $ be two $ n $-plots  in $ X $ and $ r_0\in \mathrm{dom}(P)\cap\mathrm{dom}(Q) $ such that $ P(r_0)=Q(r_0) $ and
		$ f\circ P|_U=f\circ Q|_U $, on some open neighborhood $ U \subseteq\mathrm{dom}(P)\cap\mathrm{dom}(Q) $ of $ r_0 $.
		Then $ R:=f\circ P|_U $ is a plot in $ Y $, so 
		for $ (r_0,P(r_0)) \in  R^*X $, 
		there exist a D-open neighborhood $ O\subseteq R^*X $ of $ (r_0,P(r_0)) $, a D-open neighborhood $ V\subseteq U $ of $ R^*f(O) $ and a smooth map $ \varrho:V\rightarrow U\times X $ such that $ \varrho~~\circ~~ R^*f~~(r,x)=(r,x), $ for all $ (r,x)\in O $.  
		But $  P|_U $ and $Q|_U $ both are local lift plots of $ R $ along $ f $ taking $ r_0 $ to $ P(r_0) $. By Remark \ref{rem-dimm}, we conclude that
		$ P|_W=\Pr_2\circ\rho~~|_W=Q|_W $ for some open neighborhood $ W\subseteq U $ of $ r_0 $.
	\end{proof}
		\begin{corollary}\label{cor-dimm-mfd2} 
		If  $ f:X\rightarrow Y $ is a diffeological immersion, then for given $ x\in X $, any smooth map  $ g:M\rightarrow Y $ from a manifold $ M $, and $ m\in M $ with $ g(m)=f(x)  $, there exists at most one local smooth lift  $ h:O\rightarrow X $ defined on a D-open
		neighborhood $ O\subseteq M $ of $ m $ such that $ f\circ h=g|_O $ and $ h(m)=x $.
	\end{corollary}


		\begin{proposition}\label{pro-pb-imm}
	Any immersion is a diffeological immersion.
\end{proposition}
\begin{proof}
	Suppose that $ f:X\rightarrow Y $ is an immersion of diffeological spaces and $ P:U\rightarrow Y $ is a plot.  
	Let $ (r_0,x_0) \in P^*X $ or $ P(r_0)=f(x_0) $. Recall that there exist a D-open neighborhood $ O\subseteq X $ of $ x_0 $, a D-open neighborhood $ O'\subseteq Y $ of $ f(O) $ and a smooth map $ \varrho:O'\rightarrow X $ such that $ \varrho\circ f(x)=x $ for all $ x\in O $.  
	Set $ W=(P_{\#})^{-1}(O) $ and $ V=P^{-1}(O') $, which are D-open neighborhoods of the point $ (r_0,x_0) \in P^*X $ and the subset $ P^*f~~(W) $ in $ U $, respectively.
	Define the smooth map
	$ \overline{\rho}:V\rightarrow U\times X $  by $ \overline{\rho}=(\imath,\rho\circ P|_V) $ in which $ \imath $ denotes the inclusion $ V\hookrightarrow U $.
	Then for all $ (r,x) \in W $,
	\begin{center}
		$ \overline{\rho}\circ (P^*f)~~~(r,x)=\overline{\rho}(r)=(r,\rho\circ P(r))=(r,\rho\circ f(x))=(r,x), $ 
	\end{center}
	 where we used the fact that
	$ f\circ~~ P_{\#}~~|_{W}=P\circ~~(P^*f)~~|_{W} $.
	Therefore,  $ f $ is a diffeological immersion.
\end{proof}

	We here show that the converse of the above proposition is true for maps into manifolds.
	
	\begin{proposition} 
		If $ f:X\rightarrow M $ is a diffeological immersion into a manifold $ M $, then it is an immersion.
	\end{proposition}
	\begin{proof}
		Let $ x_0\in X $ and $ \varphi:U\rightarrow M  $ be a chart around $ f(x_0) $, i.e., $ f(x_0)=\varphi(r_0) $ for some  $ r_0\in U $.
		Then for $ (r_0,x_0)\in \varphi^*X $, there is a D-open neighborhood $ O $ of $ (r_0,x_0) $ in $ \varphi^*X $, an open neighborhood $ V\subseteq U $ of $ \varphi^*f(O) $ and a smooth map $ \varrho:V\rightarrow U\times X $ such that $ \varrho\circ \varphi^*f(r,x)=(r,x) $ for all $ (r,x)\in O $.
		But $ \varphi_{\#}(O) $ is a D-open subset of $ X $. 
		In fact, let $ P:\mathrm{dom}(P)\rightarrow X $ be a plot and $ s\in \mathrm{dom}(P) $ with $ P(s)\in \varphi_{\#}(O) $. Consider the plot $ Q:(f\circ P)^{-1}(\varphi(U))\rightarrow\varphi^*X $ given by
		$ Q(r)=(\varphi^{-1}\circ f\circ P(r),P(r)) $.
		By definition, $ Q^{-1}(O) $ is open and  we have $ s\in Q^{-1}(O)\subseteq P^{-1}(\varphi_{\#}(O)) $.
		
		Thus, $ O'=\varphi_{\#}(O)\cap f^{-1}(\varphi(V)) $ is a D-open neighborhood of $ x_0 $ in $ X $ and $ \varphi(V)\subseteq M $ is a D-open neighborhood of $ f(O') $.
		We claim that the composition
		$ \Pr_2\circ\rho\circ(\varphi|_V)^{-1}:\varphi(V)\rightarrow X $ is a smooth  left inverse for $ f|_{O'} $. 
		For every $ x \in O' $, we have $ ((\varphi|_V)^{-1}\circ f(x),x)\in O $ and this yields
		\begin{center}
			$ ((\varphi|_V)^{-1}\circ f(x),x)=\rho\circ\varphi^*f~~((\varphi|_V)^{-1}\circ f(x),x)=\rho\circ(\varphi|_V)^{-1}~~\circ ~~f(x). $
		\end{center}
		Consequently, 
		\begin{center}
			$ \Pr_2\circ\rho\circ(\varphi|_V)^{-1}~~\circ ~~f(x)=x, $  
		\end{center}
		for all $ x\in O' $. This completes the proof.
	\end{proof}
	
	\begin{corollary}\label{cor-dimm-mfd}
		A map between manifolds is a diffeological immersion if and only if it is an immersion of manifolds.
	\end{corollary}

	\begin{proposition}\label{pro-dim-dimm} 
		If $ f:X\rightarrow Y $ is a diffeological immersion, then $ \mathrm{dim}(X)\leq\mathrm{dim}(Y) $.
	\end{proposition}
	\begin{proof}
		Suppose that $ \mathcal{C} $ is a covering generating family of $ Y $.
		For any plot $ Q $ belonging to $ \mathcal{C} $,
		the pullback of $ f $ by $ Q $ is an immersion and hence for each $ (r_0,x_0) \in  Q^*X $, there exist a D-open neighborhood $ O $ of $ (r_0,x_0) $ in $ Q^*X $, a D-open neighborhood $ V\subseteq \mathrm{dom}(Q) $ of $ Q^*f(O) $ and a smooth map $ \varrho_{Q,r_0,x_0}:V\rightarrow \mathrm{dom}(Q)\times X $ such that $ \varrho_{Q,r_0,x_0}\circ Q^*f(r,x)=(r,x) $ for all $ (r,x)\in O $.
		Consider the set $ f^{\bullet}\mathcal{C} $ consisting of the plots $ \Pr_2\circ \rho_{Q,r_0,x_0} $,  for all $ Q\in\mathcal{C} $ and $ (r_0,x_0) \in Q^*X $.
		
		We prove that $ f^{\bullet}\mathcal{C} $ is a covering generating family of $ X $.
		It is not hard to check that $ f^{\bullet}\mathcal{C} $ is a parametrized cover of  $ X $. Now let $ P:U\rightarrow X $ be any plot in $ X $ and $ r_0\in U $. Then 
		$ f\circ P $ is a plot in $ Y $, and we can write  $ f\circ P|_{U'}=Q\circ F $ for some $ Q\in \mathcal{C}  $, an open
		neighborhood $ U'\subseteq U $ of $ r_0 $, and a smooth map $ F $ between domains.
		Notice that
		$ (F,P|_{U'}) $
		is a smooth map in $ Q^*X $.
		For $ (F(r_0),P(r_0))\in Q^*X $, we have 
		a D-open neighborhood $ O $ of $ (F(r_0),P(r_0)) $ in $ Q^*X $ and a smooth map $ \varrho_{Q,F(r_0),P(r_0)}:V\rightarrow \mathrm{dom}(Q)\times X $ as above.
		Set $  W=(F,P)^{-1}(O) $ so that
		\begin{center}
			$ (F,P)|_{W}= \rho_{Q,F(r_0),P(r_0)}~~\circ~~ Q^*f~~\circ~~(F,P)|_{W}=\rho_{Q,F(r_0),P(r_0)} ~~\circ~~ F|_{W} $.
		\end{center}
		Consequently,		$ P|_{W}=\Pr_2 ~~\circ ~~\rho_{Q,F(r_0),P(r_0)}~~\circ~~ F|_{W} $.
		Thus, $ f^{\bullet}\mathcal{C} $ is a covering generating family for $ X $ with
		$ \mathrm{dim}(f^{\bullet}\mathcal{C})=\mathrm{dim}(\mathcal{C}) $.
		Therefore,
		\begin{align*}
			\mathrm{dim}(X)\quad  &= \quad\inf_{\mathcal{G} \in\mathsf{CGF}(X)} \mathrm{dim}(\mathcal{G})\\
			&\leq \inf_{\mathcal{C}\in\mathsf{CGF}(Y)} \mathrm{dim}(f^{\bullet}\mathcal{C})\\
			&= \inf_{\mathcal{C}\in\mathsf{CGF}(Y)} \mathrm{dim}(\mathcal{C})\\
			&= \quad\quad\mathrm{dim}(Y). 
		\end{align*}
	\end{proof}
	
	\begin{proposition}
		If $ f:X\rightarrow Y $ is a diffeological immersion, then the internal tangent map $ df_x:T_xX\rightarrow T_{f(x)}Y $ is a monomorphism at each point $ x\in X $.
	\end{proposition}
	\begin{proof}
	Let $ u=\sum_{i=1}^{k} d(P_i)_0(u_i)\in~~~ T_xX $, where $ P_i $ are plots in $ X $ centered at $ x $, for $ i=1,\cdots,k $,
	 and $ u_i\in T_0\mathrm{dom}(P_i) $. Let $ df_x(u)=\sum_{i=1}^{k} d(f\circ P_i)_0(u_i)=0 $.
	By \cite[Lemma 3.1]{CW1}, there exist plots
	$ P'_i $ in $ Y $ centered at $ f(x) $, for $ i=k+1,\cdots,n $, distinct from each other and $ f\circ P_i $'s, finitely many vectors $ v_{ij} $ 
	satisfying 
\begin{align*}
	\sum_{j}v_{ij}=\left\{
	\begin{array}{lcl}
		u_i,\quad &  \mathrm{if} &i \leq k\\
		0,\quad &  \mathrm{if} &i\geq k+1\\
	\end{array} \right.
\end{align*}
	along with plots 
	$ Q_{ij} $ in $ Y $ centered at $ f(x) $ and germs of smooth maps $ F_{ij} $ between domains such that
	$ f\circ P_i=Q_{ij}\circ F_{ij} $, for $ i=1,\cdots,k $, 
	and
	 $ P'_i=Q_{ij}\circ F_{ij} $, for $ i=k+1,\cdots,n $,  up to germs at $ 0 $.
	Moreover, for each $ Q\in \{Q_{ij} \} $, one has 
\begin{center}
		$ \displaystyle\sum_{i,j~~~\mid~~~ Q_{ij}=Q} d(F_{ij})_0(v_{ij})=0 $
\end{center}
	 in $ T_0\mathrm{dom}(Q) $.
Since $ f $ is a diffeological immersion, by an argument similar to that in the proof of Proposition \ref{pro-dim-dimm},
we obtain plots $ \rho_{ij}=\Pr_2 ~~\circ ~~\rho_{Q_{ij},0,x} $ in $ X $ centered at $ x $ such that $ P_i=\rho_{ij}\circ F_{ij} $ up to germs at $ 0 $, for $ i=1,\cdots,k $.
Also, if $ Q_{ij}=Q_{i'j'}  $, then  $ \rho_{ij}=\rho_{i'j'}  $.
Thus,
\begin{align*}
u=\displaystyle\sum_{i=1}^{k} d(P_i)_0(u_i)&=\sum_{i=1}^{n} d(\rho_{ij}\circ F_{ij})_0(\sum_{j}v_{ij})\\
&=\sum_{i,j} d(\rho_{ij})_0\circ d(F_{ij})_0(v_{ij}) \\
&=\sum_{Q\in \{Q_{ij}\}}d(\rho_Q)_0\Big{(}\sum_{i,j~~~\mid~~~ Q_{ij}=Q} d(F_{ij})_0(v_{ij})\Big{)}\\
&=\sum_{Q\in \{Q_{ij}\}}d(\rho_Q)_0(0)=0,
\end{align*}
where $ \rho_Q=\Pr_2 ~~\circ ~~\rho_{Q,0,x} $.
	\end{proof}
	
	\begin{proposition}\label{prop-immer-cons} 
		Suppose that $ f:X\rightarrow Y $ is a diffeological immersion and $ h:X'\rightarrow X $ is a D-continuous map between diffeological spaces. The map $ h $ is smooth if and only if $ f\circ h $ is smooth.
	\end{proposition}
	\begin{proof}
		It clear that if $ h $ is smooth, then $ f\circ h $ is smooth. Conversely, assume that $ f\circ h $ is smooth. To prove $ h $ is smooth, 
		let $ Q:U\rightarrow X' $ be a plot in $ X' $ and $ r_0\in U $. Then the composition $ P:=f\circ h\circ Q $ is a plot in $ Y $ and by definition, for 
		$ (r_0,h\circ Q(r_0)) \in P^*X $, there exist a D-open neighborhood $ O $ of the point $ (r_0,h\circ Q(r_0)) $ in $ P^*X $, a D-open neighborhood $ V\subseteq U $ of the set $ P^*f(O) $ and a smooth map $ \varrho:V\rightarrow U\times X $ such that $ \varrho\circ P^*f(r,x)=(r,x) $ for all $ (r,x)\in O $.
		This follows that
		$ O\subseteq\mathsf{Graph}(\Pr_2\circ\rho) $.
		Since $ h $ is D-continuous, so  is $ (\mathrm{id}_U,h\circ Q):U\rightarrow P^*X $.
		Then, $ W=(\mathrm{id}_U,h\circ Q)^{-1}(O) $
		is an open neighborhood $ W\subseteq U $ of $ r_0 $ such that $ \mathsf{Graph}(h\circ Q|_W)\subseteq O \subseteq\mathsf{Graph}(\Pr_2\circ\rho|_W) $.
		Hence $ h\circ Q|_W=\Pr_2\circ\rho~~|_W $, which follows that $ h\circ Q $ is a plot in $ X $ by D3.
	\end{proof}
	Injective diffeological immersions may not be inductions in general. 
	The following examples clear the situation.
	
	\begin{example}\label{exa-ind-dimm}
		The  smooth map 
		\begin{center}
			$ f:(-\pi,\pi)\rightarrow  \mathbb{R}^2$\quad given by\quad $t\mapsto(\sin t, \sin 2t) $
		\end{center} is an injective (diffeological) immersion which is not an induction (see \cite[Exercise 59]{PIZ}).
		On the other hand, the induction $ [0,\infty)\hookrightarrow \mathbb{R} $ is not a diffeological immersion. In fact, by \cite[Exercise 51]{PIZ}, we have
		$ \infty=\mathrm{dim}([0,\infty)) \nleq \mathrm{dim}(\mathbb{R})=1 $.
		As another example, the map $ f: \mathbb{R}\rightarrow \mathbb{R}^2$  given by $ f(t)=(t^2,t^3) $ is an induction (see, \cite[Th\'eor\`eme 1]{Jor}) but not an immersion (see also \cite{KMW} for more details).
	\end{example}
	
	\subsection{Embeddings of diffeological spaces}
	We here discuss some notions for embeddings of diffeological spaces.
Let us first recall the embeddings due to Iglesias-Zemmour.
		\begin{definition}
		(\cite[\S 2.13]{PIZ}).
		An \textbf{embedding} is an induction which is a D-embedding,  i.e., a topological embedding with respect to the D-topology. 
	\end{definition}
Following \cite{KMW}, we can also consider weak embeddings.
\begin{definition}
	A map $ f:X\rightarrow Y $ between diffeological spaces is called a \textbf{weak embedding} if it is both  an induction and a diffeological immersion.
\end{definition}

Embeddings and weak embeddings are different concepts, even in the class of manifolds (see \cite{KMW,PIZ2022}).	For instance, the map $ f:\mathbb{R}\rightarrow \mathbb{R}^2 $ with $ f(t)=(t^2,t^3) $ is an embedding, but not a weak embedding. On the other hand, the map $ g:\mathbb{R}\rightarrow \mathbb{T}^2 $ with $ g(t)=(e^{2\pi it},e^{2\pi i\alpha t})$ for some $ \alpha\notin\mathbb{Q}, $ is a weak embedding,  which is not an embedding.

Analogous to those of manifolds, we define a stronger version of embeddings of diffeological spaces.
	\begin{definition}\label{def-demb}
		A map $ f:X\rightarrow Y $ between diffeological spaces is called a \textbf{diffeological embedding} or \textbf{strong embedding} if it is a diffeological immersion which is  a D-embedding.
	\end{definition}

 \begin{proposition} 
	Any diffeological embedding $ f:X\rightarrow Y $ is an induction.
	In other words, a diffeological embedding is a diffeomorphism onto its image as a subspace.
\end{proposition}
\begin{proof}
By definition, $ f $ is a smooth injective map.
	Suppose that $ P:U\rightarrow X $ is a parametrization such that $ f\circ P $ is a plot in $ Y $. Then $ f\circ P:U\rightarrow Y $ is a D-continuous map whose image is contained in $ f(X) $. So $ f\circ P|^{f(X)}:U\rightarrow f(X) $ is D-continuous. As $ f $ is a D-embedding, $ P $ is D-continuous as well. Thus, $ P $ is a plot in $ X $ by Proposition \ref{prop-immer-cons}. 
\end{proof}

 \begin{corollary} \label{coro-embd-st}
	A map $ f:X\rightarrow Y $ between diffeological spaces is a diffeological embedding if and only if it is both an embedding and a weak embedding.
\end{corollary}
 
\begin{proposition} 
	A map between manifolds is a diffeological embedding if and only if it is an embedding of manifolds.
\end{proposition}
\begin{proof}
	This is a consequence of Corollary \ref{cor-dimm-mfd}, and the fact that the D-topology of manifolds agrees with the underlying topology.
\end{proof}


		\begin{proposition} 
	Any immersion of diffeological spaces is locally a diffeological embedding.
\end{proposition}
	\begin{proof}
	In fact, any immersion is locally a smooth map with a smooth left inversion.
\end{proof}	

	\subsubsection{Some examples}
\begin{example}(\cite{AM}).
	Let $ X $ be a diffeological space.
	For either the union power set diffeology or the strong power set diffeology on $ \mathfrak{P}(X) $, the natural map $ \imath:X\rightarrow\mathfrak{P}(X) $ defined by $ \imath(x)=\{x\} $ is a diffeological embedding.
\end{example}
Smooth maps with smooth left inversions, in particular, smooth retractions are among the simplest and most common cases of diffological embeddings.
\begin{example}
	The map $ I:X\rightarrow \mathrm{Paths}(X) $ taking any $ x\in X $ to the constant path $ \widetilde{x}:\mathbb{R}\rightarrow X  $ with the value $ x $ is a diffeological embedding.
	Let $ P:U\rightarrow X $ be a plot and let $ F:V\rightarrow \mathbb{R} $ be a smooth map between domains.
	Since $ \big{(}(I\circ P)\circledcirc F\big{)}(r,s)=I\circ P(r)(F(s))= \widetilde{P(r)}\big{(}F(s)\big{)}=P(r) $  is a plot in $ X $, $ I\circ P $ is a plot in $ \mathrm{Paths}(X)  $ and indeed, $ I $ is smooth.
	Moreover, $ \mathrm{ev}_0:\mathrm{Paths}(X)\rightarrow X, \gamma\mapsto \gamma(0)  $ is a smooth left inverse of $ I $.
	Of course, $ \mathrm{ev}_1:\mathrm{Paths}(X)\rightarrow X, \gamma\mapsto \gamma(1)  $ could be another smooth left inverse for $ I $.
\end{example}

 	\begin{example} 
		Let $ \varrho:X\rightarrow X $ be a smooth retraction of diffeological spaces, i.e., $ \varrho\circ\varrho=\varrho $ (see \cite[\S 5.13]{PIZ}). Then the restriction
		$ \varrho|_{\varrho(X)}:\varrho(X)\rightarrow X $	is 
		a diffeological embedding, while  $ \varrho|^{\varrho(X)}:X\rightarrow \varrho(X) $ is a strong subduction.
	\end{example}

	\begin{example}\label{exa-dim-R} 
	Let $ \lbrace X_i\rbrace_{i=1}^\infty $ be a countably infinite family of nonempty diffeological spaces and fix some $ \lbrace \ast_i\rbrace_{i=1}^\infty\in \prod_{i=1}^\infty X_i $. 
	Let $ k\geq 1 $ be an integer. The map
	$ \rho^k:\prod_{i=1}^\infty X_i\rightarrow\prod_{i=1}^\infty X_i $
	replacing only the $ i $th term with $ \ast_i $, for all $ i>k $,
	is a smooth retraction of diffeological spaces. 
	In particular, if $ X_i $'s are all manifolds of  dimensions greater than $ 0 $, we can say that the diffeological dimension of $ \prod_{i=1}^\infty X_i $ is infinite.
	As a special case, we have $ {\rm dim}(\mathbb{R}^{\omega})=\infty $, where $ \mathbb{R}^{\omega} $ denotes the countably infinite product of copies of $ \mathbb{R} $.
\end{example}

	\begin{example}\label{exa-gln} 
		As shown in \cite[Exercise 57]{PIZ} the inclusion $ \imath:\mathrm{GL}(n,\mathbb{R})\hookrightarrow \mathrm{Diff}(\mathbb{R}^n) $ is 
		an embedding, where $ \mathrm{GL}(n,\mathbb{R})\subseteq \mathbb{R}^{n\times n} $ is equipped with the subspace diffeology inherited from $ \mathbb{R}^{n\times n} $. Furthermore, we check that $ \imath $ is actually a diffeological embedding:
		If $ U $ is an $ n $-domain, the map $ D:\mathrm{C}^{\infty}(U,\mathbb{R}^m)\rightarrow \mathrm{C}^{\infty}(U,\mathbb{R}^{n\times m}) $ taking $ f  $ to its total derivative  $ D(f) $ is  smooth by \cite[Lemma 4.3]{CSW2014}.
		Thus, the map $ D_0:\mathrm{Diff}(\mathbb{R}^n)\rightarrow  \mathrm{GL}(n,\mathbb{R}) $ taking $ f  $ to the total derivative of $ f $ at $ 0 $ is  smooth and we have $ D_0\circ \imath(M)=M $, for all $ M\in\mathrm{GL}(n,\mathbb{R}) $.
		As a result, $ D_0:\mathrm{Diff}(\mathbb{R}^n)\rightarrow \mathrm{Diff}(\mathbb{R}^n) $
		is a smooth retraction of diffeological spaces. 
	\end{example} 

	\begin{example}\label{exa-dim}
		Let $ U $ be an $ n $-domain and fix $ r_0\in U $. For any integer $ k\geq 0 $, consider the map
		$ T^k_{r_0}:\mathrm{C}^{\infty}(U,\mathbb{R}^m)\rightarrow \mathrm{C}^{\infty}(U,\mathbb{R}^m) $
		taking any $ f=(f_1,\cdots,f_m) $ to 
		$ T^k_{r_0}(f)=(P_1,\cdots,P_m) $, where $ P_i $ is
		the $ k $ degree Taylor's polynomial map 	at $ r_0 $ restricted to $ U $,
\begin{center}
			$ P_i: x\mapsto\sum_{|\alpha|\leq k}\dfrac{D^{\alpha}f_i(r_0)}{\alpha !}(x-r_0)^{\alpha} $
\end{center}
which is understood by the multi-index notation.
	By \cite[Lemma 4.3]{CSW2014},	$ T^k_{r_0} $
		is smooth. 
		Obviously, the image of $ T^k_{r_0} $ is the space of polynomials in $ \mathbb{R}^m $ of degree $ k $ with $ n $ variables
		restricted to $ U $, denoted by $ \mathrm{Pol}_k(U,\mathbb{R}^m) $ and equipped with the subspace diffeology inherited from $ \mathrm{C}^{\infty}(U,\mathbb{R}^m) $.
		Since the $ k $ degree Taylor's polynomial of a $ k $ degree polynomial is itself, $ T^k_{r_0} $ is a smooth retraction of diffeological spaces.
	Then the inclusion $ \mathrm{Pol}_k(U,\mathbb{R}^m)\hookrightarrow \mathrm{C}^{\infty}(U,\mathbb{R}^m) $ is 
	a diffeological embedding.
	
	But $ \mathrm{Pol}_k(U,\mathbb{R}^m) $ is diffeomorphic to 
	$ \mathbb{R}^{m\times\binom{n+k}{k}} $, under the diffeomorphism
	$ \mathrm{Pol}_k(U,\mathbb{R}^m) \rightarrow \mathbb{R}^{m\times\binom{n+k}{k}} $
	taking
	$ (P_1,\cdots,P_m)\in\mathrm{Pol}_k(U,\mathbb{R}^m) $ to
	\begin{center}
$ \begin{pmatrix} 
(\dfrac{D^{\alpha}P_1(r_0)}{\alpha !})_{|\alpha|\leq k}\\
\vdots\\
(\dfrac{D^{\alpha}P_m(r_0)}{\alpha !})_{|\alpha|\leq k}\\
\end{pmatrix} $
	\end{center}
	Therefore, for every $ k\geq 0 $, we get
\begin{center}
		 $ m\times\binom{n+k}{k}\leq {\rm dim}(\mathrm{C}^{\infty}(U,\mathbb{R}^m)) $. 
\end{center}
We also have
\begin{center}
	$ c_k:=\binom{n+k}{k}=(n+1)(\frac{n}{2}+1)\cdots(\frac{n}{k}+1)\geq n+\frac{n}{2}+\cdots+\frac{n}{k}+1=ns_k+1, $
\end{center}
where
$ s_k=\sum_{i=1}^k\frac{1}{i} $.
Since the sequence	$ \{s_k\} $ is  unbounded, 
so is the sequence $ \{c_k\} $.
In conclusion, the diffeological dimension of $ \mathrm{C}^{\infty}(U,\mathbb{R}^m) $ is infinite.
    \end{example}


	\section{\'{E}tale maps in diffeology}\label{S5} 
	In this section, we deal with diffeological \'{e}tale maps and study their properties.
	First we recall \'{e}tale maps of diffeological spaces.
	\begin{definition}\label{def-eta}
		(\cite[\S 2.5]{PIZ}).
		A map $ f:X\rightarrow Y $ between diffeological spaces is \textbf{\'{e}tale} if for every $ x_0 $ in $ X $, there are D-open neighborhoods $ O\subseteq X $ and $ V\subseteq Y $ of $ x_0 $ and $ f(x_0) $, respectively, such that $ f|_O:O\rightarrow O' $ is a diffeomorphism.  
	\end{definition}

	It is trivial that	a map between manifolds is \'{e}tale if and only if it is a local diffeomorphism.
	
	\begin{example}\label{exa-op-et}	
		\begin{enumerate}
			\item[$ \blacktriangleright  $] 
			Let $ X $ be a diffeological space.
			The inclusion $ O\hookrightarrow X $ of a D-open subset $ O $ into $ X $ is a simple example of \'{e}tale maps.	
			
			\item[$ \blacktriangleright $] 
			The canonical injection $ X_i\rightarrow\bigsqcup_{i\in J} X_i $ of each component $ X_i $ of the sum space $ \bigsqcup_{i\in J} X_i $ of a family of diffeological spaces is an \'{e}tale map. 
		\end{enumerate}
	\end{example}
	
	\begin{theorem}\label{the-che} 
		(Characterizations of \'{e}tale maps).
		Let $ f:X\rightarrow Y $	be a map between diffeological spaces. The following statements are equivalent:
		\begin{enumerate}
			\item[$ (1) $] 
			$ f $ is an \'{e}tale map.
			\item[$ (2) $] 	
			$ f $ is both a  submersion and an immersion. 
			\item[$ (3) $] 
			$ f $ is a D-open immersion.
		\end{enumerate}
	\end{theorem}
	\begin{proof}
		$ (1)\Rightarrow(2)$ and $(2)\Rightarrow(3) $ are obvious. To complete the proof, we need to prove $ (3)\Rightarrow (1) $. For each $ x_0 $ in $ X $, one has a D-open neighborhood $ O\subseteq X $ of $ x_0 $, a D-open neighborhood $ O'\subseteq Y $ of $ f(O) $ and a smooth map $ \varrho:O'\rightarrow X $ such that $ \rho\circ f(x)=x $  for all $ x\in O $.
		Since $ f $ is a D-open map, $ f(O)\subseteq Y $ is D-open as well. Now the restriction $ \rho|_{f(O)} $ from the D-open subset  $ f(O)\subseteq Y $ to the D-open subset $ O\subseteq X $, is the desired inverse for $ f|_O:O\rightarrow f(O) $ and $ f|_O $ is actually a diffeomorphism.
	\end{proof}
	
	\begin{definition}\label{def-deta}
		A smooth map $ \pi:\mathcal{E}\rightarrow X $ is a \textbf{diffeological \'{e}tale map} if the pullback $ P^{*}\pi $ by every plot $ P $ in $ X $ is  \'{e}tale. In this situation, the total space $ \mathcal{E} $ is called a \textbf{diffeological \'{e}tale space} over the \textbf{base space} $X$, and
		the fiber $ \mathcal{E}_x:=\pi^{-1}(x) $ is called the \textbf{stalk} of $ \mathcal{E} $ over $ x $, for all $ x\in X $.
	\end{definition}
	
	Remark that we do not require that diffeological \'{e}tale spaces to be either surjective or injective.
	It is immediate that every \'{e}tale map is a diffeological \'{e}tale map but not conversely (see Example \ref{exa-irt} below).

	\begin{theorem}\label{the-et-chr} 
		(Characterizations of diffeological \'{e}tale maps).
		Let $ \pi:\mathcal{E}\rightarrow X $	be a map between diffeological spaces. The following statements are equivalent:
		\begin{enumerate}
			\item[ (a) ] 
			$ \pi $ is a diffeological \'{e}tale map.
			\item[ (b)] 
			$ \pi $ is a smooth map and there exists a covering generating family  $ \mathcal{C} $ of the diffeological space $ X $ such that the pullback of $ \pi $ by every plot $ P $ belonging to $ \mathcal{C} $ is \'{e}tale.
			\item[(c)] 	
			$ \pi $ is both a diffeological immersion and a diffeological submersion. 
			\item[(d)] 
			$ \pi $ is a diffeological immersion whose pullbacks by plots in $ X $ are D-open maps.
			\item[(e)] 
			$ \pi $ is both a diffeological injection and a diffeological submersion. 
			\item[(f)] 
			$ \pi $ is a smooth map and for given $ \xi\in \mathcal{E} $, any smooth map  $ f:M\rightarrow X $ from a manifold $ M $, and $ m\in M $ with $ f(m)=\pi(\xi)  $, there exists only one local smooth lift  $ l:O\rightarrow X $ defined on a D-open
			neighborhood $ O\subseteq M $ of $ m $ such that $ \pi\circ l=f|_O $ and $ l(m)=\xi $.
			\item[(g)] 
			$ \pi $ is a smooth map and for any $ \xi\in \mathcal{E} $, any plot $ P:U\rightarrow X $, 
			and any $ r_0\in U $ satisfying $ P(r_0)=\pi(\xi)  $,
			there exist an open
			neighborhood $ V\subseteq U $ of $ r_0 $ such that $ \pi $ admits only one  local lift plot $ L:V\rightarrow\mathcal{E} $ of $ P|_V $ along $ \pi $ with $ L(r_0)=\xi $. 
			
		\end{enumerate}
	\end{theorem}
	\begin{proof}
		 (a) $ \Rightarrow $ (b)  is easily verified.
		 (b) $ \Rightarrow $ (c)  is a consequence of Theorem \ref{the-che}.
		But (c) $ \Rightarrow $ (d) follows from the fact that submersion are D-open maps.
		By Theorem \ref{the-che} and Proposition \ref{pro-imm-inj},  (d) $ \Rightarrow $ (e)   is obtained.
		By Corollary \ref{cor-dsub-mfd2} and Proposition \ref{pro-dinj-mfd},   (e) $ \Rightarrow $ (f)   is immediate.
		 (f) $ \Rightarrow $ (g)  is trivial.
		
		To prove  (g) $ \Rightarrow $ (a), let $ P:U\rightarrow X $ be any plot and $ (r_0,\xi)\in P^{*} \mathcal{E} $. Then $ P(r_0)=\pi(\xi)  $ and 
		by assumption, there exist an open neighborhood $ V\subseteq U $ of $ r_0 $ such that $ \pi $ admits only one  local lift plot $ L:V\rightarrow\mathcal{E} $ of $ P|_V $ along $ \pi $ with $ L(r_0)=\xi $. 
		In particular, $ \pi $ is a diffeological injection and by Proposition \ref{pro-inj-lift}, $ \mathsf{Graph}(L) $ is a D-open neighborhood of $ (r_0,\xi) $ in $ P^*X $.
		The restriction  $ P^*\pi|_{\mathsf{Graph}(L)}:\mathsf{Graph}(L)\rightarrow V $ is a diffeomorphism with the inverse $ V\rightarrow\mathsf{Graph}(L), r\mapsto(r,L(r)) $.
		Hence  $ P^{*}\pi $ is \'{e}tale.
	\end{proof}
	
	\begin{corollary} \label{cor-tang}
		\begin{enumerate}
			\item[ (i) ]
			The internal tangent map of a diffeological \'etale map at each point is an isomorphism.
			\item[ (ii) ] 
			On a manifold, \'{e}tale spaces and diffeological \'{e}tale spaces are the same.
			\item[ (iii) ] 
			The stalks of a diffeological \'{e}tale space are all discrete subspaces of the total space.
			\item[ (iv) ] 
			Any diffeological \'{e}tale map is a D-open map. 
			\item[ (v) ] 
			Any injective diffeological \'{e}tale map is a diffeological embedding. 
		\end{enumerate}
	\end{corollary}

	\begin{remark}\label{rem-tng}
		The converse to Corollary \ref{cor-tang}(i) dose not hold.
		By \cite[Proposition 3.4]{CW}, the internal tangent map of smooth map $ f:\mathbb{R}_2^3\rightarrow\mathbb{R}^3, ~~x\mapsto x $  is an isomorphism, but 
		$ f $ is not a diffeological \'etale map (see Example \ref{exa-lasa}).
	\end{remark}
	\begin{remark}
		Part (ii) of  Corollary \ref{cor-tang} still holds for the non-manifold $ \mathbb{R}\bigsqcup\mathbb{R}^2 $.
		But Example \ref{exa-irt} shows this  is not true in general.
	\end{remark}
	\begin{proposition}\label{pro-etal-dim}
		If $ \pi:\mathcal{E}\rightarrow X $ is a diffeological \'{e}tale map, then 
		\begin{center}
			$ \mathrm{dim}(\mathcal{E})=\mathrm{dim}(\pi(\mathcal{E}))\leq\mathrm{dim}(X) $.
		\end{center}
	\end{proposition}
	\begin{proof}
		By virtue of Corollary \ref{coro-dim-dsub} and Proposition \ref{pro-dim-dimm}, we have		$ \mathrm{dim}(\pi(\mathcal{E}))\leq\mathrm{dim}(\mathcal{E})\leq\mathrm{dim}(X) $.
		Let $ \mathcal{C} $ be a  covering generating family of $ \pi(\mathcal{E})  $.
		Then the collection $ \mathcal{C}' $ of local lift plots of elements of $ \mathcal{C} $ along $ \pi $ constitutes a covering generating family for $ \mathcal{E} $ with $ \mathrm{dim}(\mathcal{C}')=\mathrm{dim}(\mathcal{C}) $ and consequently, $ \mathrm{dim}(\pi(\mathcal{E}))\geq \mathrm{dim}(\mathcal{E}) $.
		To prove this assertion, let $ P:U\rightarrow X $ be any plot and $ r_0\in U $. Then $ \pi\circ P|_V=Q\circ F $  for some $ Q\in \mathcal{C} $ and a smooth map $ F $ defined on an open neighborhood $ V\subseteq U $ of $ r_0 $.
		On the other hand, since $ \pi $ is a diffeological \'{e}tale map, there is a local lift plot $ L:W\rightarrow \mathcal{E} $ defined on an open neighborhood $ W\subseteq \mathrm{dom}(Q) $ of $ F(r_0) $ such that $ \pi\circ L=Q|_W $ and $ L(F(r_0))=P(r_0) $.
		This implies that $ \pi\circ L\circ F|_{V'}=Q\circ F|_{V'} $ on an open neighborhood $ V'\subseteq V $ of $ r_0 $.
		Thus, $ \pi\circ P|_{V'}=\pi\circ L\circ F|_{V'} $, which yields $ P|_{V'}=L\circ F|_{V'} $.
	\end{proof}

	\begin{remark}\label{rem-etal-dim}
	 In Proposition \ref{pro-etal-dim}, the right-hand side inequality may be strict, even for \'{e}tale maps.
	 As an example, the canonical injection $ \mathbb{R}\rightarrow\mathbb{R}\bigsqcup\mathbb{R}^2 $ is an \'{e}tale map, however, 
	 	\begin{center}
	 	$ \mathrm{dim}(\mathbb{R})=1 <2=\mathrm{dim}(\mathbb{R}\bigsqcup\mathbb{R}^2) $.
	 \end{center}
	\end{remark}

	\begin{proposition} 
		A map between diffeological spaces is a diffeological covering map if and only if it is both a diffeological \'{e}tale map and a diffeological fiber bundle.
	\end{proposition}
	\begin{proof}
		Because locally trivial bundles with discrete fibers are \'{e}tale, the proof is easy.
	\end{proof}
	
	\begin{corollary}
		If $ \pi:\mathcal{E} \rightarrow X  $ is a diffeological covering map, then $ \mathrm{dim}(\mathcal{E})=\mathrm{dim}(X) $, and at any point $ \xi\in \mathcal{E} $, $ T_{\xi}\mathcal{E} $ is isomorphic to $ T_{\pi(\xi)}X $.
	\end{corollary}
	\begin{example}
		Let $ (X,\mathcal{D}) $ be a diffeological space.
		The map $ f:X \rightarrow \mathcal{D} $ taking every point $ x $ to the $ 0 $-plot $ \textbf{x} $ is a diffeological \'{e}tale map.
		As a result, $ \mathrm{dim}(X)\leq\mathrm{dim}(\mathcal{D}) $.
		
        To observe this,  we only show that $ f $ is a diffeological submersion,
        because it is injective.
		Let $ \rho:U \rightarrow \mathcal{D} $ be a plot and  $r_0\in U, x_0\in X $ be such that $ f(x_0)=\rho(r_0) $.
		By definition, there exist an open neighborhood $  V\subseteq U $ of $ r_0 $
		such that $  \rho(r) $ is a $ 0 $-plot for all $ r\in V $, and $ r\mapsto \rho(r)(0) $ defined on $ V $ is a plot in $ X $.
		Take $ L: V\rightarrow X $ to be $ L(r)=\rho(r)(0) $ so that $ f\circ L=\rho|_V $ and $ L(r_0)=x_0 $.
	\end{example}

		\begin{proposition}\label{p-gitm-etl}
		If $ \pi:\mathcal{E}\rightarrow X $ is a diffeological \'{e}tale map, then 
		$ d\pi:T^H\mathcal{E}\rightarrow T^HX $ is also a diffeological \'{e}tale map.
	\end{proposition}
	\begin{proof}
		$ d\pi $ is smooth by virtue of Proposition \ref{pro-gtm-sm}(d). To see that $ d\pi $ is a diffeological submersion, assume that $ dP $ is the global tangent map of a plot $ P:U\rightarrow X $, and let $ (\xi_0,v_0)\in T^H\mathcal{E} $ and $ (r_0,w_0)\in TU $ be such that $ d\pi(\xi_0,v_0)=dP(r_0,w_0) $, which means that 
		$ (\pi(\xi_0),d\pi_{\xi_0}(v_0))=(P(r_0),dP_{r_0}(w_0)) $.
		Because $ \pi $ itself is a diffeological submersion, there are an open neighborhood $ V\subseteq U $ of $ r_0 $ and a local lift plot $ L:V\rightarrow \mathcal{E} $ satisfying $ \pi\circ L=P|_V $ and  $ L(r_0)=\xi_0 $. 
		Then 
		$ d\pi\circ dL=dP|_{TV} $. Moreover,  $ d\pi_{\xi_0}\circ dL_{r_0}(w_0)=dP_{r_0}(w_0)=d\pi_{\xi_0}(v_0) $.
		Since $ d\pi_{\xi_0} $ is injective, we get 
		$ dL_{r_0}(w_0)=v_0 $, which follows that
		$ dL(r_0,w_0)=(\xi_0,v_0) $.
		
		Now suppose that $ P=(P_1,P_2) $ and $ Q=(Q_1,Q_2) $ are two $ n $-plots in $ T^H\mathcal{E} $. Let $ r_0\in \mathrm{dom}(P)\cap\mathrm{dom}(Q) $ with $ P(r_0)=Q(r_0) $,
		and there exists an open neighborhood $ U $ of $ r_0 $ in $ \mathrm{dom}(P)\cap\mathrm{dom}(Q) $ such that	$ d\pi\circ P|_U=d\pi\circ Q|_U $.
		By definition, 
		\begin{center}
			$ (\pi\circ P_1(r),d\pi_{P_1(r)} (P_2(r)))=(\pi\circ Q_1(r),d\pi_{Q_1(r)} (Q_2(r))) $
		\end{center}
		for all $ r\in U $.
		Since $ \pi $ is a diffeological injective, $ P_1|_V=Q_1|_V $ for some open neighborhood $ V $ of $ r_0 $ in $U $.
		On the other hand, because $ d\pi_{P_1(r)}=d\pi_{Q_1(r)}  $ is injective, $ P_2(r)=Q_2(r)  $ for all $ r\in V $.
		Therefore, $ P|_V=Q_V $.
		This completes the proof.
	\end{proof}
	\begin{lemma}\label{lem-dinj-dsub}
		Suppose we are given 
		a commutative diagram of smooth maps as the following:
		$$
		\xymatrix{
			X \ar[dr]_{g} \ar[rr]^{f} & &Y \ar[dl]^{h}\\
			& Z  & 
		}
		$$
		\begin{enumerate}
			\item[$ (1)  $] 
			If $ g $ is a diffeological submersion and  $  h $ is a diffeological injection, then  $ f $ is a diffeological submersion.
			\item[$ (2)  $] 
			If $ g $ is a diffeological injection, then  $ f $ is a diffeological injection.
			\item[$ (3)  $] 
			If $ g $ is a diffeological immersion, then  $ f $ is a diffeological immersion.
			\item[$ (4)  $] 
			If $ g $ is a diffeological \'{e}tale map and  $  h $ is a diffeological injection, then  $ f $ is a diffeological \'{e}tale map.
			\item[$ (5)  $] 
			If $ f $ is surjective and  $ g $ is a diffeological submersion, then $ h $ is a diffeological submersion.
			\item[$ (6)  $] 
			If $ f $ is a surjective diffeological submersion and $ g $ is a diffeological injection, then $ h $ is a diffeological injection.
			\item[$ (7)  $] 
			If $ f $ is surjective diffeological submersion and  $ g $ is a diffeological immersion, then $ h $ is a diffeological immersion.
			\item[$ (8)  $] 		If $ f $ is surjective diffeological submersion  and  $ g $ is a diffeological \'{e}tale maps, then $ h $ is a diffeological \'{e}tale map.
		\end{enumerate}	
	\end{lemma}
	\begin{proof}
		To prove $ (1) $, suppose that $ P:U\rightarrow  Y $ is a plot and let $ x_0 \in X $ and $ r_0\in U $ be such that $ f(x_0)=P(r_0) $.
		Then $ h\circ P $ is a plot in $ Z $ and $ g(x_0)=h\circ P(r_0) $. Since $ g $ is a diffeological submersion, there are 
		an open neighborhood $ V\subseteq U $ of $ r_0 $ and a local lift plot $ L:V\rightarrow X $ satisfying $ g\circ L=h\circ P|_V $ and  $ L(r_0)=x_0 $. Then,
		$ h\circ f\circ L=h\circ P|_V $ and $ f\circ L(r_0)=f(x_0)=P|_V(r_0) $.
		But $ h $ is a diffeological injection, so $ f\circ L|_W=P|_W $ on some open neighborhood $ W\subseteq V $ of $ r_0 $, which means that $ f $ is a diffeological submersion.
		
		Part $ (2) $ is straightforward.

		To verify $ (3) $, let $ P:U\rightarrow  Y $ is a plot and $ (r_0,x_0) \in P^*X $. Since $ g $ is a diffeological immersion, 
		there exist a D-open neighborhood $ O $ of $ (r_0,x_0) $ in $ (h\circ P)^*X $, an open neighborhood $ V\subseteq U $ of $ [(h\circ P)^*g]~~~(O) $ and a smooth map $ \varrho:V\rightarrow U\times X $ such that $ \varrho\circ~~~ [(h\circ P)^*g]~~~(r,x)=(r,x) $ for all $ (r,x)\in O $.
		But we have a natural inclusion $ \imath:P^*X\hookrightarrow (h\circ P)^*X $ for which	$ P^*f=[(h\circ P)^*g]\circ\imath $. 
		So $ O':=\imath^{-1}(O)=(P^*X)\cap O $ is a D-open neighborhood of $ (r_0,x_0) $ in $ P^*X $ and for every $ (r,x)\in O' $,
		\begin{center}
			$ \rho\circ P^*f(r,x)=\rho~~~\circ ~~~[(h\circ P)^*g]~~~\circ~~~\imath~~(r,x)=(r,x) $.
		\end{center}
		
		Part $ (4) $ is a consequence of $ (1) $ and $ (2) $.
		
		To establish $ (5) $, let $ P:U\rightarrow Z $ be a plot and let $ y_0\in Y $ and $ r_0\in U $ be such that $ h(y_0)=P(r_0) $.
		By assumption, there exists an element $ x_0\in X $ with $ f(x_0)=y_0 $, and consequently, we get
		$ g(x_0)=P(r_0) $. So there are 
		an open neighborhood $ V\subseteq U $ of $ r_0 $ and a local lift plot $ L:V\rightarrow X $ satisfying $ g\circ L=P|_V $ and  $ L(r_0)=x_0 $.  Then $ h\circ (f \circ L)=P|_V $ and $ f \circ L(r_0)=y_0 $. 
		
		To show $ (6) $, assume that $ P $ and $ Q $ are two $ n $-plots in $ Y $, $ r_0\in \mathrm{dom}(P)\cap\mathrm{dom}(Q) $, let $ P(r_0)=Q(r_0) $
		and there exists an open neighborhood $ U $ of $ r_0 $ in $\mathrm{dom}(P)\cap\mathrm{dom}(Q) $ such that	$ h\circ P|_U=h\circ Q|_U $.
		By hypothesis, there exists an element $ x_0\in X $ with $ f(x_0)=P(r_0)=Q(r_0) $. Also, we can find local lift plots $ L:V\rightarrow X $ and $ L':V\rightarrow X $ of 
		$ P $ and $ Q $ along $ f $, respectively, defined on an open neighborhood $ V\subseteq U $ of $ r_0 $ such that $ L(r_0)=x_0=L'(r_0) $. Then, 
		\begin{center}
			$ g\circ L=h\circ f\circ L =h\circ P|_V=h\circ Q|_V=h\circ f\circ L'=g\circ L' $.
		\end{center}
		Since $ g $ is a diffeological injection, this follows that $ L|_W=L'|_W $ and so $ P|_W=Q|_W $ on some open neighborhood $ W\subseteq V $ of $ r_0 $.
		
		To prove $ (7) $, let $ P:U\rightarrow Z $ be a plot and $ (r_0,y_0) \in P^*Y $.
		Since $ f $ is surjective, there exists an element $ x_0\in X $ with $ f(x_0)=y_0 $ so that
		$ g(x_0)=P(r_0) $ or $ (r_0,x_0)\in P^*X $.
		Because $ g $ is a diffeological immersion, 
		there exist a D-open neighborhood $ O $ of $ (r_0,x_0) $ in $ P^*X $, an open neighborhood $ V\subseteq U $ of $ P^*g(O) $ and a smooth map $ \varrho:V\rightarrow U\times X $ such that $ \varrho\circ P^*g(r,x)=(r,x) $ for all $ (r,x)\in O $.
		Now consider $ \theta:P^*X\rightarrow P^*Y, ~~~(r,x)\mapsto(r,f(x)) $, which is a well-defined smooth map
		and  $ P^*g=P^*h\circ\theta $.
		It is not so hard to check that $ \theta $ is a diffeological submersion and so a D-open map.
		Thus, $ \theta(O) $ is a D-open neighborhood of $ (r_0,y_0) $ in $ P^*Y $.
		Consider also smooth map $ \overline{\theta}:U\times X\rightarrow U\times Y, ~~~ \overline{\theta}(r,x)\mapsto(r,f(x)) $, which is an extension of $ \theta $.
		Now if $ (r,y)\in \theta(O) $ with $ \theta(r,x)=(r,y) $, for some $ (r,x)\in O $, then
		\begin{center}
			$ (\overline{\theta}\circ\rho)\circ P^*h(r,y)=(\overline{\theta}\circ\rho)\circ P^*h\circ\theta(r,x)=(\overline{\theta}\circ\rho)\circ P^*g(r,x)=\overline{\theta}(r,x)=(r,y) $.
		\end{center}
		
		Part $ (8) $ is a result of $ (5) $ and $ (6) $.
	\end{proof}
	
	\begin{corollary}
		If $ f:X\rightarrow (Y,\mathcal{D}) $ is a diffeological submersion (respectively, diffeological injection, diffeological immersion, diffeological \'{e}tale map) and 
$ id:(Y,\mathcal{D}')\rightarrow (Y,\mathcal{D}) $	is smooth, then $ f:X\rightarrow (Y,\mathcal{D}') $ is a diffeological submersion (respectively, diffeological injection, diffeological immersion, diffeological \'{e}tale map).
	\end{corollary}
	
	We denote by $ \mathsf{Etale}(X) $ the collection of diffeological \'{e}tale spaces on a diffeological space $ X $  as a full subcategory of the comma category $ \mathsf{Diff}/ X $, or equivalently, that of comma sheaves (see \cite{DA}).
	That is, a morphism $ \phi $ between diffeological \'{e}tale spaces  $ \mathcal{E} $  and $\mathcal{E}'  $ on  $ X $ is a commutative diagram
	$$
	\xymatrix{
		\mathcal{E} \ar[dr]_{\pi} \ar[rr]^{\phi} & & \mathcal{E}' \ar[dl]^{\pi'}\\
		& X  & 
	}
	$$
	of smooth maps,
	or equivalently, a smooth stalk preserving map $ \phi:\mathcal{E}\rightarrow \mathcal{E}'  $. 
	
	As a consequence of Lemma \ref{lem-dinj-dsub}(4) and the fact that diffeological \'{e}tale maps are D-open, we can now rephrase the following proposition:
	\begin{proposition}
		A morphism $ \phi:\mathcal{E}\rightarrow \mathcal{E}'  $ of diffeological \'{e}tale maps is a diffeological \'{e}tale map itself.
		In particular, for a diffeological \'{e}tale map $ \pi:\mathcal{E}\rightarrow X $ and a subspace $ \mathcal{E}'\subseteq \mathcal{E} $,
		the restriction $ \pi|_{\mathcal{E}'}:\mathcal{E}'\rightarrow X $ is a diffeological \'{e}tale map if and only if $ \mathcal{E}' $ is a D-open subspace of $ \mathcal{E} $.
	\end{proposition}
	
	Diffeological \'{e}tale maps are stable under compositions and pullbacks by smooth maps.
%
%
%
	\begin{proposition}\label{pro-com-pb-et}

	\begin{enumerate}
		\item[ (i) ] A map  $ \pi:\mathcal{E}\rightarrow X $ is a diffeological \'{e}tale map if and only if
		 the composition $ \pi\circ f $
		is a diffeological \'{e}tale map, for all diffeological \'{e}tale maps $ f $ into $\mathcal{E} $. 
		
		\item[ (ii) ] A map  $ \pi:\mathcal{E}\rightarrow X $ is a diffeological \'{e}tale map if and only if
		the pullback $ f^*\pi $
		is a diffeological \'{e}tale map, for all smooth maps $ f $ into $ X $.

	\end{enumerate}
	
\end{proposition}
	\begin{proof}
		The proof is straightforward.
	\end{proof}
	Since 
	$ \mathsf{Etale}(X) $
	is a category with a terminal object $ \mathrm{id}_X $ and
	with all pullbacks, it has all finite limits.
	In particular,
	the fiber product $ \mathcal{E}\times_X\mathcal{E}'  $ of two diffeological \'{e}tale spaces $ \mathcal{E} $ and $\mathcal{E}'  $ on a diffeological space $ X $ is a diffeological \'{e}tale space on $ X $.
	\begin{proposition}
		Let $ \pi:\mathcal{E}\rightarrow X $ be a smooth map and let $ f:X'\rightarrow X $ be a surjective diffeological submersion.
		Then, $ \pi $ is a diffeological \'{e}tale map if and only if the pullback of $ \pi $ by $ f $ is a diffeological \'{e}tale map.
	\end{proposition}
	\begin{proof}
		Suppose that $ f^*\pi $ is a diffeological \'{e}tale map. To show that $ \pi $ is a diffeological \'{e}tale map, 
		let $ P:U\rightarrow X $ be a plot in $ X $, $ \xi\in \mathcal{E}, r_0\in U $ be such that $ \pi(\xi)=P(r_0) $.
		But $ f $ is surjective, so we can find $ x_0\in X' $ with $ f(x_0)=P(r_0) $, and since $ f $ is a diffeological submersion, there exists a 
		local lift plot $ L:V\rightarrow X' $ satisfying $ f\circ L=P|_V $ and  $ L(r_0)=x_0 $, where $ V\subseteq U $ is an open neighborhood of $ r_0 $. But $ (x_0,\xi)\in f^*\mathcal{E} $ and $ f^*\pi(x_0,\xi)=L(r_0) $.
		Thus, there are an open neighborhood $ W\subseteq V $ of $ r_0 $ and a
		local lift plot $ L':W\rightarrow f^*\mathcal{E} $ such that $ f^*\pi\circ L'=L|_W $ and  $ L'(r_0)=(x_0,\xi) $.
		Then $ f_{\#}\circ L' $ is a local lift plot $ P $ along $ \pi $ with $ f_{\#}\circ L'(r_0)=\xi $.
		To complete the proof, we need to verify that $ \pi $ is a diffeological injection.
		Assume that $ P $ and $ Q $ are two $ n $-plots in $ \mathcal{E} $, $ r_0\in \mathrm{dom}(P)\cap\mathrm{dom}(Q) $, $ P(r_0)=Q(r_0) $
		and there exists an open neighborhood $ U $ of $ r_0 $ in $\mathrm{dom}(P)\cap\mathrm{dom}(Q) $ such that	$ \pi\circ P|_U=\pi\circ Q|_U $.
		Because  $ f $ is subduction, there is a plot $ L $ defined on an open neighborhood $ V\subseteq U $ of $ r_0 $ such that locally lifts $ \pi\circ P|_U=\pi\circ Q|_U $ along $ f $. Now $ (L,P|_V) $ and $ (L,Q|_V) $ are 
		two $ n $-plots in $ f^*\mathcal{E} $ with $ (L,P|_V)(r_0)=(L,Q|_V)(r_0) $ and $ f^*\pi\circ(L,P|_V)=f^*\pi\circ(L,Q|_V) $.
		Then $ (L_W,P|_W)=(L|_W,Q|_W) $ and so $ P|_W=Q|_W $ on some open neighborhood $ W\subseteq V $ of $ r_0 $.
		Therefore, $ \pi $ is a diffeological \'{e}tale map. The converse follows from Proposition \ref{pro-com-pb-et}$ (ii) $.
	\end{proof}
	
		\subsection{Diffeological \'{e}tale spaces on an orbifold}
		
	The aim of this section is to investigate diffeological \'{e}tale spaces on a diffeological orbifold. 
	Let us begin by recalling the definition of diffeological orbifolds from \cite{IKZ}.
	\begin{definition}\label{def-ord}  
		A \textbf{diffeological $ n $-orbifold} is a diffeological space $ \mathcal{O} $ such that each point $ x\in \mathcal{O} $ has a D-open neighborhood $ x\in \widetilde{U}\subseteq \mathcal{O} $ for which there is a diffeomorphism $ \widetilde{\varphi}: U/G\rightarrow \widetilde{U} $, where  $ U/G $ is the quotient
		of some action of a finite subgroup $ G\subseteq\mathrm{GL}(n,\mathbb{R}) $
		on a connected $ G $-invariant $ n $-domain $ U $.
	\end{definition} 
	
	In practice,  we can work with $ G $-invariant diffeological submersions $ \varphi:U\rightarrow \mathcal{O} $ defined on connected $ G $-invariant $ n $-domains, meaning that, for all $ x,y\in U $,
	\begin{center}
		$ \varphi(x)=\varphi(y)\qquad$ if and only if $ \qquad Gx=Gy $.
	\end{center}
	In fact,  the composition $ \varphi=\widetilde{\varphi}\circ \pi $ is a $ G $-invariant diffeological submersion from a connected domain onto a D-open set, where $ \pi:U\rightarrow U/G $ denotes the quotient map.

	On the other hand, by Proposition \ref{pro-subd-smd}, a $ G $-invariant diffeological submersion $ \varphi:U\rightarrow \mathcal{O} $ from a connected domain induces  a diffeomorphism $ \widetilde{\varphi}: U/G\rightarrow \varphi(U) $
	onto a D-open subset $ \varphi(U) $ of $ \mathcal{O} $,  making the following diagram commutative:
	\begin{displaymath}
		\xymatrix{
			& \mathcal{O} \\
			U \ar[ur]^{\varphi}\ar[r]_{\pi} & U/G\ar[u]_{\widetilde{\varphi}}  \\
		}
	\end{displaymath}
	We call $ (U, G, \varphi) $ an \textbf{orbifold chart}. 
	The collection  of these orbifold charts constitutes a covering generating family of $ \mathcal{O} $.
	In this situation, there is no need to verify the compatibility of orbifold charts. Indeed, if $ (U, G, \varphi) $ and $ (V, H, \psi) $ are two orbifold charts at $ x\in \mathcal{O} $ with $ \varphi(r)=x=\psi(s) $, then by \cite[Lemmas 17 and 23]{IKZ}, we can find an equivariant \'etale map $ h:U'\rightarrow V $ between domains such that $ h(r)=s $ and the following diagram commutes:
	\begin{displaymath}
		\xymatrix{
			& \mathcal{O} & \\
			U' \ar[ur]^{\varphi|_{U'}}\ar[rr]_{h} & & V\ar[ul]_{\psi}  \\
		}
	\end{displaymath}
	In other words, a diffeological $ n $-orbifold is constructed by gluing together $ n $-domains along \textit{equivariant} \'etale maps, unique up to an action of the finite linear groups.
	
		Notice that the internal tangent map of an orbifold chart is not an isomorphism in general. 
		For instance, consider the quotient map  $ q:\mathbb{R}\rightarrow \mathbb{R}/O(1) $ discussed in Example \ref{exa-inj-dis} as an orbifold chart.
		In fact by the computation of \cite[Example 3.24]{CW}, we get
			\begin{center}
					$ 1={\rm{dim}}(T_0\mathbb{R})\neq {\rm{dim}}(T_{[0]}\mathbb{R}/O(1))=0 $.
				\end{center}
		As a corollary, an orbifold chart is not necessarily a diffeological \'{e}tale map.
	
	Now we state  our main theorem: 
	
	\begin{theorem}\label{the-orb}
		A diffeological \'{e}tale space on a diffeological orbifold  $ \mathcal{O} $ is a  diffeological orbifold.
	\end{theorem}
	\begin{proof}
		Let $ \mathcal{E} $ be a diffeological \'{e}tale space on $ \mathcal{O} $ through a diffeological \'{e}tale map $ \pi:\mathcal{E}\rightarrow \mathcal{O}$.
		%
		Let $ \xi\in\mathcal{E} $ and $ (U,G,\varphi) $ be an orbifold chart around $ \pi(\xi) $ with $ \varphi(r_0)= \pi(\xi) $, for some $ r_0\in U $.
		So there exists an open neighborhood $ V\subseteq U $ of $ r_0 $ and a  unique local lift plot $ L:V\rightarrow\mathcal{E} $  such that $ \pi\circ L=\varphi|_V $ and $ L(r_0)=\xi $. 
		Now we construct an orbifold chart for  $ \mathcal{E} $ around $ \xi $ from $ L $.
		According to \cite[Lemma B.1.3]{Sch}, one can find an open $ G $-stable
		neighborhood $ W\subseteq V $ of $ r_0 $ such that $ G_W=G_{r_0} $. Because $ \pi $ is diffeological injection, we have $ L\circ g|_{W}=L|_{W} $, for all $ g\in G_W $.
		Thus, for $ x,y\in W $ if $ G_Wx=G_Wy $, then $ L(x)=L(y) $.
		Conversely, for $ x,y\in W $ if $ L(x)=L(y) $, then $ \varphi(x)=\varphi(y) $ and so there is an element $ g\in G $ satisfying $ g(y)=x $. 
		But $ W $ is $ G $-stable and $ x\in g(W)\cap W\neq\varnothing $, so $ g(W)=W $ and  $ g\in G_W $.
		Hence $ L|_{W} $ is $ G_W $-invariant.
		In addition, by Lemma \ref{lem-dinj-dsub}(1), $ L|_{W}:W\rightarrow \mathcal{E} $ is a diffeological submersion.
		Actually, $ L|_{W}:W\rightarrow L(W) $ is a $ G $-invariant diffeological submersion from a connected domain onto a D-open set and
		$ (W,G_W,L|_{W}) $ is an orbifold chart for  $ \mathcal{E} $ at $ \xi $.
	\end{proof}	
	
	\begin{corollary} 
		A diffeological covering space of a diffeological orbifold is a diffeological orbifold.
	\end{corollary}
	
	\begin{remark} 
		Sheaves on an orbifold $ \mathcal{O} $ (in the sense of Moerdijk and Pronk \cite{MP}) are just topological version of diffeological \'{e}tale spaces on $ \mathcal{O} $.
		In fact, from a sheaf $ \mathcal{S} $ on $ \mathcal{O} $, we can patch together the sheaves $ \pi_{\widetilde{U}}:\mathcal{S}_{\widetilde{U}}\rightarrow\widetilde{U} $ given on the domains of orbifold charts along embeddings of orbifold charts (see \cite[p. 5]{MP}) so that we obtain a natural continuous map $ \pi:\bigsqcup_{(\widetilde{U},G,\varphi)}\mathcal{S}_{\widetilde{U}}/\varrho\longrightarrow \mathcal{O} $ taking a class $ [(\widetilde{U},G,\varphi,s)] $ to $ \varphi(\pi_{\widetilde{U}}(s)) $,
		where $ (\widetilde{U},G,\varphi,s)~~~ \varrho ~~~(\widetilde{V},H,\psi,s') $ if and only if there are
		embeddings of orbifold charts
		\begin{displaymath}
			\xymatrix{
				(\widetilde{U},G,\varphi)& (\tilde{W},K,\chi)\ar[l]_{\lambda}\ar[r]^{\mu}	& (\widetilde{V},H,\psi ) \\
				\mathcal{S}_{\widetilde{U}}\ar[u]^{\pi_{\widetilde{U}}} & \mathcal{S}_{\widetilde{W}}\ar[u]_{\pi_{\widetilde{W}}}\ar[l]^{\lambda_{\#}}\ar[r]_{\mu_{\#}} & \mathcal{S}_{\widetilde{V}}\ar[u]_{\pi_{\widetilde{V}}}  \\
			}
		\end{displaymath}
		and $ s''\in\mathcal{S}_{\widetilde{W}} $ such that $ \lambda_{\#}(s'')=s $ and $ \mu_{\#}(s'')=s' $.
		In this situation, the pullback of $ \pi $ by any orbifold chart is a usual sheaf. 
	\end{remark}
	
	\section{Diffeological \'{e}tale manifolds}\label{S6}
	This section is devoted to an investigation of diffeological \'{e}tale manifolds. After exploring their basic properties, we prove versions of the rank and implicit function theorems and also, the fundamental theorem on flows for diffeological \'{e}tale manifolds. 
	
	First of all, we give our motivating example.

	\begin{example}\label{exa-irt}
		The diffeological covering map $ \pi:\mathbb{R} \rightarrow \mathbb{T}_{\alpha} $, $ \alpha\notin\mathbb{Q}  $,   is a diffeological \'{e}tale map. However, it is not an \'{e}tale map;
		for otherwise, there would be an open subspace of $ \mathbb{R} $ with a trivial subspace topology which is impossible.
		Consequently, $ T_x\mathbb{T}_{\alpha}$ is isomorphic to $\mathbb{R} $ at any point $ x\in \mathbb{T}_{\alpha} $, a similar result to \cite[Example 3.23]{CW}.
		More generally, let $ \Gamma $ be a discrete subgroup generating $ \mathbb{R}^n $ and $ \mathbb{T}^n_{\Gamma}=\mathbb{R}^n/\Gamma $ be its quotient space. The quotient map $ \mathbb{R}^n\rightarrow\mathbb{T}^n_{\Gamma} $ is a  diffeological \'{e}tale map by \cite[Exercise 133]{PIZ}.
		As a result, $ \mathrm{dim}(\mathbb{T}^n_{\Gamma})=n $
		and at any point $ x\in \mathbb{T}^n_{\Gamma} $, we have $ T_x\mathbb{T}^n_{\Gamma}\cong \mathbb{R}^n $.
	\end{example}
We can generalize this situation and introduce diffeological \'{e}tale $ n $-manifolds.
	\begin{definition}\label{def-ord}  
		A diffeological space $ \mathcal{M} $ is said to be a \textbf{diffeological \'{e}tale $ n $-manifold} if there exists a parametrized cover $ \mathfrak{A} $, called an \textbf{atlas}, for $ \mathcal{M} $ consisting of diffeological \'{e}tale maps 
		from $ n $-domains into $ \mathcal{M} $.
		We call the elements of $ \mathfrak{A} $ \textbf{diffeological \'{e}tale charts}.
		In this situation, $ \mathfrak{A} $ is actually a covering generating family for $ \mathcal{M} $ and $ \mathrm{dim}(\mathcal{M})=n $.
	\end{definition} 

	\begin{example}
	Diffeological manifolds are diffeological \'{e}tale manifolds but not conversely.
	The irrational torus $ \mathbb{T}_{\alpha} $, where $ \alpha\notin\mathbb{Q}  $,  is a diffeological \'{e}tale $ 1 $-manifold which is not a diffeological manifold.
		More generally, $ \mathbb{T}^n_{\Gamma}=\mathbb{R}^n/\Gamma $ is a diffeological \'{e}tale $ n $-manifold.
	\end{example}

	\begin{example}
Any open subset of a diffeological \'{e}tale $ n $-manifold itself is  a diffeological \'{e}tale $ n $-manifold
in a natural way. 
	\end{example}

	\begin{example}(Product \'{e}tale manifolds).
	Suppose $ \mathcal{M}_1,\dots,\mathcal{M}_k $ are diffeological \'{e}tale manifolds of dimensions $ n_1,\dots,n_k $, respectively. The product space $ \mathcal{M}_1\times\cdots\times\mathcal{M}_k $
	is a diffeological \'{e}tale manifold of dimension $ n_1+\cdots+n_k $ with
	diffeological \'{e}tale charts of the form $ \varphi_1\times\cdots\times\varphi_k $, where $ \varphi_i $ is a diffeological \'{e}tale chart of  $ \mathcal{M}_i $, $ i=1,\dots,k $.
\end{example}
	\begin{proposition}
		A diffeological space $ \mathcal{M} $ is a diffeological \'{e}tale manifold if and only if there exists a parametrized cover $ \mathfrak{A}  $ of $ \mathcal{M} $
		such that the evaluation map $ ev:\mathsf{Nebula}(\mathfrak{A})\rightarrow \mathcal{M}, ~~~(P,r)\mapsto P(r) $ is a diffeological \'{e}tale  map.
	\end{proposition}
	\begin{proof}
		The proof is straightforward.
	\end{proof}
	To see the relationship between diffeological \'{e}tale charts, let $ \varphi:U\rightarrow \mathcal{M} $ and $ \psi:V\rightarrow \mathcal{M} $ be two diffeological \'{e}tale charts in $ \mathcal{M} $ with $ \varphi(r)=\psi(s) $. Then we can find a unique smooth map $ h:U'\rightarrow V $ defined on an open neighborhood $ U'\subseteq U $ of $ r $ such that $ h(r)=s $ and the following diagram commutes:
	\begin{displaymath}
		\xymatrix{
			&  \mathcal{M}  & \\
			U' \ar[ur]^{\varphi|_{U'}}\ar[rr]_{h} & & V\ar[ul]_{\psi}  \\
		}
	\end{displaymath}
	By Lemma \ref{lem-dinj-dsub}(4), $ h $ is  an \'{e}tale map, which is  \textit{uniquely} determined with this property.
	While diffeological \'{e}tale charts might not be locally injective, we can say that a diffeological \'{e}tale $ n $-manifold is obtained by gluing together $ n $-domains along \'{e}tale maps $ h $.

	\begin{proposition} 
		A diffeological \'{e}tale space on a diffeological \'{e}tale $ n $-manifold is again a diffeological \'{e}tale $ n $-manifold.
	\end{proposition}
	\begin{proof}
		The proof is straightforward.
	\end{proof}
	\begin{proposition}\label{pro-loc-cont}
		Let $ \mathcal{M} $ and $ \mathcal{N} $ be diffeological \'{e}tale manifolds.
		A smooth map $ f:\mathcal{M}\rightarrow \mathcal{N} $ is locally constant if and only if $ df_x=0 $,
		for all $ x\in\mathcal{M} $.
	\end{proposition}
	\begin{proof}
		In view of Proposition \ref{pro-cons}, we just need to prove that if $ df_x=0 $ for all $ x\in\mathcal{M} $,  then $ f $ is locally constant.
Let $ x\in\mathcal{M} $ and fix any diffeological \'{e}tale chart $ \varphi:U\rightarrow \mathcal{M} $ with $ \varphi(r)=x $  for some $ r\in U $.
The composition $ f\circ\varphi $ is a plot in $ \mathcal{N} $. Hence there are
diffeological \'{e}tale chart $ \psi:V\rightarrow \mathcal{N} $ and a smooth map $ F:U'\rightarrow V $ defined on an open neighborhood of $ r\in U'\subseteq U $ such that $ f\circ\varphi|_{U'}=\psi\circ F $.
Then
$ 0=df_{\varphi(s)}\circ d\varphi_{s}=d\psi_{F(s)}\circ dF_s $ for all $ s\in U' $.
Since $d\psi_{F(s)} $ is an isomorphism, $ dF_s=0 $ for all $ s\in U' $. 
Therefore, $ F $ is constant on an open neighborhood of $ r $, which yields that  $ f $ is also constant on a D-open neighborhood of $ x $.
    \end{proof}

\begin{proposition} 
	Let  $ \mathcal{M} $ be a diffeological \'{e}tale $ n $-manifold. 
	\begin{enumerate}
		\item[ (a)]
		At each point $ x $ of $ \mathcal{M} $, the internal tangent space  $ T_{x}\mathcal{M} $ is isomorphic to $ \mathbb{R}^n $. 
		\item[(b)]
		The internal tangent bundle $ T^H\mathcal{M} $ is a diffeological \'{e}tale $ 2n $-manifold.
		\item[(c)]
		The smooth projection $ \pi_{\mathcal{M}}:T^H\mathcal{M} \rightarrow\mathcal{M}  $ is a diffeological fiber bundle.
	\end{enumerate}
\end{proposition}
\begin{proof} (a)
	By virtue of Corollary \ref{cor-tang}(i),  consider a diffeological \'{e}tale chart 
	with $ x $ in its image and take differential.
	
	(b) By Proposition \ref{p-gitm-etl}, 
	the global tangent map $ d\varphi $ of a diffeological \'{e}tale chart $ \varphi $ is a diffeological \'{e}tale map defined on a $ 2n $-domain, so that  $ T^H\mathcal{M}  $ is a diffeological \'{e}tale $ 2n $-manifold. 
	
	(c) Let $ \varphi:U\rightarrow\mathcal{M}  $ be a diffeological \'{e}tale chart.
	Consider the map $ \alpha:U\times\mathbb{R}^n\rightarrow \varphi^*T^H\mathcal{M} $ defined by $ \alpha(r,v)=(r,d\varphi(r,v)) $ that makes the
	following diagram commute:
		\begin{displaymath}
		\xymatrix{
		U\times\mathbb{R}^n	\ar[rd]_{pr_1}\ar[rr]^{\alpha} & & \varphi^*T^H\mathcal{M}\ar[ld]^{\varphi^*\pi_{\mathcal{M}}} \\
		 & U &   \\
		}
	\end{displaymath}
	Obviously, $ \alpha $ is smooth, and one can see that it is bijective as well.
	In fact, $ \beta:\varphi^*T^H\mathcal{M}\rightarrow U\times\mathbb{R}^n $ given by $ \beta(r,(x,v))=(r,d\varphi_r^{-1}(v)) $ is the inverse of $ \alpha $.
	
	Now we show that $ \alpha $ is an induction, which implies that $ \alpha $ is a diffeomorphism and $ \varphi^*T^H\mathcal{M} $ is locally trivial. 
	Let $ (F_1,F_2):V\rightarrow U\times\mathbb{R}^n $ be a parametrization such that $ \alpha\circ(F_1,F_2)=(F_1,d\varphi\circ(F_1,F_2)) $ is a plot in $ \varphi^*T^H\mathcal{M} $. 
	So $ F_1 $ is a plot in $ U $ and $ d\varphi\circ(F_1,F_2) $ is a plot in $ T^H\mathcal{M} $. 
	Fix $ r_0\in V $. Since $ d\varphi $ is a diffeological \'{e}tale map, for $ (F_1,F_2)(r_0)\in U\times\mathbb{R}^n $, we can find a local lift plot $ (L_1,L_2):W\rightarrow U\times\mathbb{R}^n $ defined on some open neighborhood of $ r_0 $ in $ V $ with
	$ (L_1(r_0),L_2(r_0))=(F_1(r_0),F_2(r_0))$ and
	\begin{center}
		$ d\varphi\circ(L_1,L_2)(r)=d\varphi\circ(F_1,F_2)(r) $ 
	\end{center}
	or
	\begin{center}
		$ (\varphi\circ L_1(r), d\varphi_{L_1(r)}(L_2(r)))=(\varphi\circ F_1(r),d\varphi_{F_1(r)}(F_2(r))), $ 
	\end{center}
	for all $ r\in W $.
	From diffeological injectivity of $ \varphi $, we conclude that $ L_1|_{W'}=F_1|_{W'} $.
	Putting these results together, we conclude that
\begin{center}
		$ d\varphi_{F_1(r)}(L_2(r)))=d\varphi_{L_1(r)}(L_2(r)))=d\varphi_{F_1(r)}(F_2(r)), $
\end{center}
	for all $ r\in W' $.
	Because $ d\varphi_{F_1(r)} $ is injective, 
	$ L_2(r)=F_2(r) $ for all $ r\in W' $.
	Indeed, $ F_2|_{W'} $ is a plot.
	Therefore, $ (F_1,F_2)|_{W'} $ is a plot in $ U\times\mathbb{R}^n $. By D3, $ (F_1,F_2) $ itself is  a plot, as desired.	
\end{proof}


	\subsection{The rank theorem}
	
		\begin{definition}
		We say that a smooth map $ f:X\rightarrow Y $ between  diffeological spaces has \textbf{constant internal
			rank} $ k $ if at each point $ x\in X $, the rank of the linear map $ df_x:T_xX\rightarrow T_{f(x)}Y $ is equal to $ k $.
	\end{definition}
	
	\begin{theorem}\label{the-Rank} 
		(Rank theorem for diffeological \'{e}tale manifolds).
		Let $ \mathcal{M} $ and $ \mathcal{N} $ be diffeological \'{e}tale manifolds of dimensions $ m $ and $ n $, respectively, and let
		$ f:\mathcal{M}\rightarrow \mathcal{N} $ be a smooth map with constant internal
		rank $ k $. Then for each $ x\in\mathcal{M} $, there exist diffeological \'{e}tale maps $ \varphi:U\rightarrow \mathcal{M} $ around $ x $ and $ \psi:V\rightarrow \mathcal{N} $ around $ f(x) $ defined on an $ m $-domain and an $ n $-domain, respectively, such that $ f\circ\varphi=\psi\circ \widehat{F} $, in which $ \widehat{F} $ is in the form $ \widehat{F}(x_1,\cdots,x_k,x_{k+1},\cdots,x_m)=(x_1,\cdots,x_k,0,\cdots,0) $.
	\end{theorem}
	\begin{proof}
		Take any diffeological \'{e}tale chart $ \varphi:U\rightarrow \mathcal{M} $ with $ \varphi(r)=x $  for some $ r\in U $.
		Then $ f\circ\varphi $ is a plot in $ \mathcal{N} $. So we have 
		$ f\circ\varphi|_{U'}=\psi\circ F $, for some diffeological \'{e}tale chart $ \psi:V\rightarrow \mathcal{N} $ and a smooth map $ F:U'\rightarrow V $ defined on an open neighborhood of $ r\in U'\subseteq U $.
		Taking differentials of both sides at the
		last equality, we find that
		$ df_{\varphi(s)}\circ d\varphi_{s}=d\psi_{F(s)}\circ dF_s $ for all $ s\in U' $.
		Since $  d\varphi_{s}$ and $d\psi_{F(s)} $ are isomorphisms, and $ df_{\varphi(s)} $ has constant internal
		rank $ k $, so is $ dF_s $. That is, $ F $ has constant rank $ k $. By the rank theorem of domains, one can find diffeomorphisms $ h':W'\rightarrow U' $ and $ h:W\rightarrow V $ onto open neighborhoods of $ r $ and $ F(r) $, respectively, such that $ F\circ h'=h\circ \widehat{F} $ in which $ \widehat{F} $ is in the form $ \widehat{F}(x_1,\cdots,x_k,x_{k+1},\cdots,x_m)=(x_1,\cdots,x_k,0,\cdots,0) $.
		Consequently, we have $ f\circ(\varphi\circ h')=(\psi\circ h)\circ \widehat{F} $, so $ \varphi\circ h' $ and $\psi\circ h $ are desired diffeological \'{e}tale maps.
	\end{proof}

	\begin{corollary} 
		Let $ \mathcal{M} $ and $ \mathcal{N} $ be diffeological \'{e}tale manifolds of dimensions $ m $ and $ n $, respectively, and let
		$ f:\mathcal{M}\rightarrow \mathcal{N} $ be a smooth map.
		\begin{enumerate}
			\item[ (a) ]
			If $ df_x:T_x\mathcal{M}\rightarrow T_{f(x)}\mathcal{N} $ is an epimorphism, at each point $ x\in\mathcal{M} $, then there exist diffeological \'{e}tale maps $ \varphi:U\rightarrow \mathcal{M} $ around $ x $ and $ \psi:V\rightarrow \mathcal{N} $ around $ f(x) $ defined on an $ m $-domain and an $ n $-domain, respectively, such that $ f\circ\varphi=\psi\circ \widehat{F} $, in which $ \widehat{F} $ is in the form $ \widehat{F}(x_1,\cdots,x_n,x_{n+1},\cdots,x_m)=(x_1,\cdots,x_n) $.
			\item[(b)] 
			If $ df_x:T_x\mathcal{M}\rightarrow T_{f(x)}\mathcal{N} $ is a monomorphism, at each point $ x\in\mathcal{M} $, there exist diffeological \'{e}tale maps $ \varphi:U\rightarrow \mathcal{M} $ around $ x $ and $ \psi:V\rightarrow \mathcal{N} $ around $ f(x) $ defined on an $ m $-domain and an $ n $-domain, respectively, such that $ f\circ\varphi=\psi\circ \widehat{F} $, in which $ \widehat{F} $ is in the form $ \widehat{F}(x_1,\cdots,x_m)=(x_1,\cdots,x_m,0,\cdots,0) $.
		\end{enumerate}
	\end{corollary}

	\begin{theorem}\label{the-IFT} 
		Let $ \mathcal{M} $ and $ \mathcal{N} $ be diffeological \'{e}tale manifolds of dimensions $ m $ and $ n $, respectively, and let
		$ f:\mathcal{M}\rightarrow \mathcal{N} $ be a smooth map.
		\begin{enumerate}
			\item[(i)]
			If $ df_x:T_x\mathcal{M}\rightarrow T_{f(x)}\mathcal{N} $ is an epimorphism, at each point $ x\in\mathcal{M} $, then $ f $ is a diffeological submersion.
			\item[(ii)] 
			If $ df_x:T_x \mathcal{M} \rightarrow T_{f(x)}\mathcal{N} $ is a monomorphism, at each point $ x\in\mathcal{M} $, then $ f $ is a diffeological immersion.
			\item[(iii)] 
			If $ df_x:T_x\mathcal{M}\rightarrow T_{f(x)}\mathcal{N} $ is an isomorphism, at each point $ x\in\mathcal{M} $, then $ f $ is a diffeological \'{e}tale map.
		\end{enumerate}
	\end{theorem}
	\begin{proof}
		First of all, we notice that in each case, $ f $ has constant internal
		rank so that there exist diffeological \'{e}tale maps $ \varphi:U\rightarrow \mathcal{M} $ around $ x $ and $ \psi:V\rightarrow \mathcal{N} $ around $ f(x) $  defined on an $ m $-domain and $ n $-domain, respectively, with $ \varphi(s)=x $ and $ \psi(s')=f(x) $ such that $ f\circ\varphi=\psi\circ F $ and $ F(s)=s' $, where $ F $ is a smooth map between domains with constant internal
		rank equal to that of $ f $.
		
		If $ df_x $ is an epimorphism at each point, then $ \psi\circ F $ is a diffeological submersion. Hence by Lemma \ref{lem-dinj-dsub}(5),
		$ f $ is locally a diffeological submersion and so $ f $ itself is a diffeological submersion.
		If $ df_x $ is a monomorphism at each point, then $ \psi\circ F $ is a diffeological immersion. Therefore,
		$ f $ is locally a diffeological immersion by Lemma \ref{lem-dinj-dsub}(7), so part  (ii) is obtained.
		Part  (iii)  is a direct consequence of parts (i)  and (ii).
	\end{proof}
	Theorem \ref{the-IFT}(iii) can be understood as a generalized version of the inverse function theorem to diffeological \'{e}tale manifolds.
	\begin{remark}
			Notice that the rank theorem for diffeological \'{e}tale manifolds the does not assert that a diffeological immersion is locally injective. 
\end{remark}

\subsection{The implicit function theorem}
Before we state the implicit function theorem formally, we need to prove a lemma.

	\begin{lemma}\label{lem-rank}
	Let $ \mathcal{M},\mathcal{N} $ and $ \mathcal{L} $ be diffeological \'{e}tale manifolds and let $ f:\mathcal{M}\times\mathcal{N}\rightarrow \mathcal{L} $ be a smooth map.
	For each $ x_0\in\mathcal{M} $, define $ f^{x_0}:\mathcal{N}\rightarrow \mathcal{L}  $ by $ f^{x_0}(y)=f(x_0,y) $.
	Then for every $ (x_0,y_0)\in \mathcal{M}\times\mathcal{N}  $, 
	\begin{center}
		$ \mathsf{rank}~d(\Pr_1,f)_{(x_0,y_0)}=\mathrm{dim}(\mathcal{M})+\mathsf{rank}~d\big{(}f^{x_0}\big{)}_{y_0}, $
	\end{center}
	where $ \Pr_1:\mathcal{M}\times\mathcal{N}\rightarrow \mathcal{M} $ is the projection on the first factor.
\end{lemma} 
\begin{proof}
	Fix any $ (x_0,y_0)\in \mathcal{M}\times\mathcal{N} $, and take diffeological \'{e}tale charts $ \varphi:U\rightarrow \mathcal{M} $ and $ \psi:V\rightarrow \mathcal{N} $ with $ \varphi(r_0)=x_0 $ and $ \psi(s_0)=y_0 $, for some $ r_0\in U $ and $ s_0\in V $.
	Then
	$ f\circ(\varphi\times\psi) $
	is a plot in $ \mathcal{N} $. 
	We can find a diffeological \'{e}tale chart $ \chi:W\rightarrow \mathcal{N} $ and a smooth map $ F: U'\times V'\rightarrow W $  defined on an open neighborhood of $ (r_0,s_0)\in U'\times V'\subseteq U\times V $ such that
$ f\circ(\varphi\times\psi)|_{U'\times V'}=\chi\circ F $.
Thus,
$ f^{x_0}\circ\psi|_{V'}=\chi\circ F^{r_0} $, where $ F^{r_0}:s\mapsto F(r_0,s) $.
Taking differentials, 
\begin{center}
	$ \mathsf{rank}~d\big{(}f^{x_0}\big{)}_{y_0}=\mathsf{rank}~(\dfrac{\partial F}{\partial s})_{(r_0,s_0)} $.
\end{center}
On the other hand,  we have
\begin{center}
	$ (\Pr_1,f)\circ(\varphi\times\psi)|_{U'\times V'}=(\varphi\times\chi)\circ(\pi_1,F) $,
\end{center} 
where $ \pi_1:U'\times V'\rightarrow U' $ is the projection on the first factor.
Taking again differentials of both sides, we find that
\begin{align*}
	\mathsf{rank}~d({\rm{Pr}}_1,f)_{(x_0,y_0)}&=\mathsf{rank}~d(\pi_1,F)_{(r_0,s_0)}\\
	&=\mathsf{rank} 
	\begin{pmatrix} 
		I_{\mathrm{dim}(\mathcal{M})}  & 0 \\
		&\\
		(\frac{\partial F}{\partial r})_{(r_0,s_0)} & (\frac{\partial F}{\partial s})_{(r_0,s_0)}      
	\end{pmatrix}\\
	&=\mathrm{dim}(\mathcal{M})+\mathsf{rank}~(\frac{\partial F}{\partial s})_{(r_0,s_0)}
	=\mathrm{dim}(\mathcal{M})+\mathsf{rank}~d\big{(}f^{x_0}\big{)}_{y_0},
\end{align*}
where $ I_{\mathrm{dim}(\mathcal{M})} $ denotes the identity matrix of dimension equal to $ \mathrm{dim}(\mathcal{M}) $.
\end{proof} 
Now we can state:
	\begin{theorem}\label{the-ImFT} 
		(Implicit function theorem for diffeological \'{e}tale manifolds).
		Suppose that $ \mathcal{M},\mathcal{N} $ and $ \mathcal{L} $ are diffeological \'{e}tale manifolds and let $ f:\mathcal{M}\times\mathcal{N}\rightarrow \mathcal{L} $ be a smooth map.
		For every $ (x,y)\in \mathcal{M}\times\mathcal{N}  $, $ d\big{(}f^{x}\big{)}_{y}  $ is an isomorphism if and only if
		the track map
		\begin{center}
			$ (\Pr_1,f):\mathcal{M}\times\mathcal{N}\rightarrow \mathcal{M}\times\mathcal{L},\quad (x,y)\mapsto (x,f(x,y)) $
		\end{center}
		is a diffeological \'{e}tale map.
	\end{theorem} 
	\begin{proof}
	This follows directly from Theorem \ref{the-IFT} and Lemma \ref{lem-rank}.
    \end{proof}


\begin{remark}
When $ \mathcal{M} $ and $ \mathcal{L} $ are usual manifolds,
under the hypotheses of Theorem \ref{the-ImFT}, 
for every initial condition $ (x_0,y_0)\in \mathcal{M}\times \mathcal{N} $,
we can find a D-open neighborhood  $ O\times O' \subseteq \mathcal{M}\times \mathcal{L} $ of $ (x_0,f(x_0,y_0)) $ and a unique smooth map $ \sigma : O\times O' \rightarrow \mathcal{N} $ such that $ f(x,\sigma(x,z))=z $,  
for every $ (x,z) \in O\times O' $ and $ \sigma(x_0,f(x_0,y_0))=y_0 $.  
In particular, if we take $ z_0=f(x_0,y_0) $,
we can find a \textit{unique} smooth map $ h:O\rightarrow \mathcal{N} $ with the property that
%
$ f(x,h(x))=z_0 $ for all $ x\in O $, and $ h(x_0)=y_0 $, just as in the usual case.
Thus, the classical formulation of the implicit function theorem is recovered.
\end{remark}

	\subsubsection{The generalized implicit function theorem and smooth set-valued maps}
One can easily generalize the implicit function theorem by replacing diffeological submersions and epimorphisms, and relaxing the uniqueness requirement.

\begin{proposition}\label{the-GImFp} 
	Suppose that $ \mathcal{M},\mathcal{N} $ and $ \mathcal{L} $ are diffeological \'{e}tale manifolds and let $ f:\mathcal{M}\times\mathcal{N}\rightarrow \mathcal{L} $ be a smooth map.
	For every $ (x,y)\in \mathcal{M}\times\mathcal{N}  $, $ d\big{(}f^{x}\big{)}_{y}  $ is an epimorphism if and only if
	the track map
	\begin{center}
		$ (\Pr_1,f):\mathcal{M}\times\mathcal{N}\rightarrow \mathcal{M}\times\mathcal{L},\quad (x,y)\mapsto (x,f(x,y)) $
	\end{center}
	is a diffeological submersion.
\end{proposition}

We mention a connection between  the  generalized version of the implicit function theorem (Proposition \ref{the-GImFp}) and smooth set-valued maps in the context of diffeology.
Consider a  smooth two-variable map 
$ f:X\times Y\rightarrow Z $ between diffeological spaces and the equation $ f(x,y)=z $ for given $ (x,z)\in X\times Z $.
The solutions of this equation may not be unique. 
So this defines a set-valued map
	$ \phi:X\times Z\rightarrow\mathfrak{P}_{u}(Y) $  with
	$ \phi(x,z)=\{ y\in Y \mid f(x,y)=z \}, $
where $  \mathfrak{P}_{u}(Y) $ denotes the power set	$  \mathfrak{P}(Y) $ endowed with the union power set diffeology (see \cite{AM}).
The question then naturally arises when the map $ \phi $ is smooth.
It is straightforward to check that
	the set-valued map $ \phi:X\times Z\rightarrow\mathfrak{P}_{u}(Y)  $
	is smooth if and only if 
	the track map
	\begin{center}
		$ (\Pr_1,f):X\times Y\rightarrow X\times Z,\quad (x,y)\mapsto (x,f(x,y)) $
	\end{center}
	is a diffeological submersion.
	Thanks to Proposition \ref{the-GImFp}, in the case of diffeological \'{e}tale manifolds   the next proposition is immediate.
	
	\begin{proposition} 
		Suppose that $ \mathcal{M},\mathcal{N} $ and $ \mathcal{L} $ are diffeological \'{e}tale manifolds and let $ f:\mathcal{M}\times\mathcal{N}\rightarrow \mathcal{L} $ be a smooth map.
		The set-valued map
\begin{center}
			$ \phi:\mathcal{M}\times \mathcal{L}\rightarrow\mathfrak{P}_{u}(\mathcal{N}) \qquad $ given by $ \qquad\phi(x,z)=\{ y\in \mathcal{N} \mid f(x,y)=z \} $
\end{center}
		is smooth
		if and only if
		 $ f^{x}:\mathcal{N}\rightarrow \mathcal{L}   $ is a diffeological submersion, for every $ x\in \mathcal{M}  $.
	\end{proposition} 

This reformulation of the implicit function theorem in the setting of diffeology may be useful in the analysis of the smooth dependence on initial conditions of weak solutions of functional equations arising from partial differential equations (see especially the treatment by J-P. Magnot \cite{Mag}).

	\subsection{The fundamental theorem on flows}
	We here discuss vector fields, integral curves, and flows in diffeology and show an application of them.
	

\begin{definition}
	Let $ X $ be a diffeological space.
If $ \gamma:I\rightarrow X $ is a curve defined on an $ 1 $-domain, then we have the commutative diagram
$$
\xymatrix{
	TX\ar[r]^{\pi_X}	&  X \\
	TI\ar[u]^{d\gamma}\ar[r]^{\pi_I} & I \ar[u]_{\gamma}\ar@/^0.3cm/[l]^{\frac{d}{dt}}\\
}
$$
where $ \frac{d}{dt}:I\rightarrow TI $ is the vector field sending $ s\in I $ to the standard coordinate basis vector in $ T_s\mathbb{R} $.
The \textbf{velocity} of $ \gamma $ at $ s $, denoted by $ \gamma'(s) $,
is defined by $ \gamma'(s):=d\gamma(\frac{d}{dt}(s)) $.
	An \textbf{integral curve} of a vector field $ \Lambda:X\rightarrow TX $ on $ X $  is a smooth curve $ \gamma:I\rightarrow X $ such that for all $ s\in I $,
	$ \gamma'(s)=\Lambda(\gamma(s)) $.
\end{definition}

\begin{lemma}\label{lem-1}
	The map $ \alpha:1$-$\mathrm{Plots}_0(X)\rightarrow T^HX $, $ \gamma \mapsto \gamma'(0) $, is smooth.
\end{lemma}
\begin{proof}
	Let $ \rho:U\rightarrow 1$-$\mathrm{Plots}_0(X) $ be a plot in $ 1$-$\mathrm{Plots}_0(X) $ and $ r_0\in U $.
	By definition, there exist an open neighborhood $  V\subseteq U $ of $ r_0 $
	and a number $ \epsilon>0 $ such that $ (-\epsilon,\epsilon) \subseteq\mathrm{dom}(\rho(r)) $ for all $ r\in V $, and $ \overline{\rho}:(r,s)\mapsto \rho(r)(s)$ defined on $ V\times (-\epsilon,\epsilon) $ is a plot in $ X $.
	One can see that the following diagram commutes:
	$$
	\xymatrix{
		1-\mathrm{Plots}_0(X)\ar[r]^{\alpha} &  T^HX \\
		V\ar[u]^{\rho|_{V}}\ar[r]_{(\lambda,\mu)\qquad} & T(V\times (-\epsilon,\epsilon)) \ar[u]^{d\overline{\rho}}\\
	}
	$$
	where 
	$ \lambda:V\rightarrow V\times (-\epsilon,\epsilon), r\mapsto(r,0) $ and
	$ \mu:V\rightarrow \mathbb{R}^n\times \mathbb{R}, r\mapsto(0,\frac{d}{dt}(0)) $, $ n=\mathrm{dim}(V) $.
	In fact, for all $ r\in V $ we have
	\begin{align*}
		\alpha\circ\rho|{V}(r) &= \rho(r)'(0)\\
		&= \Big{(}\rho(r)(0),d\rho(r)_{0}(\frac{d}{dt}(0))\Big{)}\\
		&= \Big{(}\rho(r)(0),d(\overline{\rho}\circ i_r)_0(\frac{d}{dt}(0))\Big{)}\\
		&= \Big{(}\rho(r)(0),d\overline{\rho}_{(r,0)}\circ d(i_r)_0(\frac{d}{dt}(0))\Big{)}\\
		&= \Big{(}\rho(r)(0),d\overline{\rho}_{(r,0)}(0,\frac{d}{dt}(0))\Big{)}\\
		&= \Big{(}\overline{\rho}\circ\lambda(r),d\overline{\rho}_{\lambda(r)}(\mu(r))\Big{)}\\
		&=	d\overline{\rho}\circ (\lambda,\mu)(r),\\
	\end{align*}
	in which
	$ i_r:(-\epsilon,\epsilon)\rightarrow V\times (-\epsilon,\epsilon), s\mapsto(r,s) $.
	This follows that $ \alpha\circ\rho|_{V} $ is a plot in $ T^HX $.
\end{proof}

\begin{definition}
	A smooth map $ Fl:X\rightarrow 1$-$\mathrm{Plots}_0(X)  $ is said to be a \textbf{flow} on $ X $
	if it satisfies the following \textbf{group laws}: 
	\begin{enumerate}
		\item[$ \bullet  $] 
		For all $ x\in X $, $ Fl(x)(0)=x $,
		\item[$ \bullet  $] 
		If $ s\in\mathrm{dom}(Fl(x)) $ and $ t\in\mathrm{dom}\Big{(}Fl\big{(}Fl(x)(s)\big{)}\Big{)} $ such that $ s+t\in\mathrm{dom}(Fl(x)) $, 
		then 
\begin{center}
			$ Fl\big{(}Fl(x)(s)\big{)}(t)=Fl(x)(s+t) $.
\end{center}
	\end{enumerate}
A flow on $ X $ is called a \textbf{complete}, 
	whenever for all $ x\in X $, $ \mathrm{dom}(Fl(x))=\mathbb{R} $. 
	Denote by  $ \mathsf{Flows}(X) $ the space of flows on $ X $ equipped with the subspace diffeology inherited from $ \mathrm{C}^{\infty}(X,1$-$\mathrm{Plots}_0(X)) $.
\end{definition}

\begin{proposition}
	Let $ Fl:X\rightarrow$  $ 1 $-$\mathrm{Plots}_0(X) $ be a flow.
	Then the map $ \Lambda:X\rightarrow T^HX $, taking $ x\in X $ to $ \big{(}Fl(x)\big{)}'(0) $, is a vector field on $ X $, called the \textbf{infinitesimal generator} of $ Fl $.
	Moreover, for each $ x\in X $, $ Fl(x) $  is an integral curve of  $ \Lambda $.
\end{proposition}
\begin{proof}
	First of all, $ \Lambda $ is smooth by Lemma \ref{lem-1}. Also, we have
	\begin{center}
		$ \pi_X\circ \Lambda(x)=\pi_X \Big{(}\big{(}Fl(x)\big{)}'(0)\Big{)}=\pi_X \Big{(}d(Fl(x))(\frac{d}{dt}(0))\Big{)}= Fl(x)(0)=x.  $
	\end{center}
	Thus, $ \Lambda $ is a vector field on $ X $.

	Next we show that $ Fl(x) $  is an integral curve of  $ \Lambda $.
	Fix any $ s\in\mathrm{dom}(Fl(x)) $.
	Since  $ 0\in\mathrm{dom}\Big{(}Fl\big{(}Fl(x)(s)\big{)}\Big{)} $ and $ \tau_s:\mathbb{R}\rightarrow\mathbb{R}, t\mapsto t+s $ is a continuous  map taking $ 0 $ to $ s $,
	there is a sufficiently small number $ \epsilon>0 $ such that for all $ t\in(-\epsilon,\epsilon) $, we have $ t\in\mathrm{dom}\Big{(}Fl\big{(}Fl(x)(s)\big{)}\Big{)} $ and  $ \tau_s(t)=t+s\in \mathrm{dom}(Fl(x)) $.
	By the group law, for all $ t\in(-\epsilon,\epsilon) $,
\begin{center}
		 $ Fl\big{(}Fl(x)(s)\big{)}(t)=Fl(x)(t+s)=\big{(}Fl(x)\circ\tau_s\big{)}(t) $.
\end{center}
	As $ d\tau_s $ takes the standard coordinate basis vector of $ T_0(\mathbb{R}) $ to that of $ T_s(\mathbb{R}) $, we can write
	\begin{align*}
		\Lambda\big{(}Fl(x)(s)\big{)} &= \big{(}Fl(Fl(x)(s))\big{)}'(0)\\
		&= \big{(}Fl(x)\circ\tau_s\big{)}'(0)\\
		&= d(Fl(x)\circ\tau_s)(\frac{d}{dt}(0))\\
		&= dFl(x)\circ d\tau_s(\frac{d}{dt}(0))\\
		&= dFl(x) (\frac{d}{dt}(s))\\
		&= \big{(}Fl(x)\big{)}'(s)).\\
	\end{align*}
\end{proof}
\begin{proposition}
	The map $ \beta:\mathsf{Flows}(X)\rightarrow \mathfrak{X}(X) $ taking a flow to its infinitesimal generator is smooth.
\end{proposition}
\begin{proof}
	In fact, $ \beta $ maps  a flow
	$ Fl:X\rightarrow$  $ 1 $-$\mathrm{Plots}_0(X) $ to the composition  $ \alpha\circ Fl $, where $ \alpha $ is given in Lemma \ref{lem-1}. Since the composition operator is smooth by \cite[\S 1.59]{PIZ}, $ \beta $ is automatically smooth.
\end{proof}

Now we show that on a diffeological \'{e}tale manifold, every vector field is the infinitesimal generator of a unique flow. 

\begin{theorem}\label{the-flow} (Fundamental theorem on flows for diffeological \'{e}tale manifolds).
	Let $ \mathcal{M} $ be a diffeological \'{e}tale manifold and $ \Lambda:\mathcal{M}\rightarrow T^H\mathcal{M} $ be a vector field on $ \mathcal{M} $.
	\begin{enumerate}
	\item[(a)] 
	For each $ x\in\mathcal{M} $, there exists a unique maximal integral curve $ \gamma:I\rightarrow\mathcal{M}, $ passing through $ x $ at $ 0 $, in the sense that it cannot be extended to an integral curve on any larger open interval.
	\item[(b)] 
	The map
	$ Fl:\mathcal{M}\rightarrow 1$-$\mathrm{Plots}_0(\mathcal{M}) $ taking any $ x\in\mathcal{M} $ to the unique maximal integral curve passing through $ x $ at $ 0 $ is a flow on $ \mathcal{M} $, and $ \bigcup_{x\in\mathcal{M}}\{x\}\times\mathrm{dom}(Fl(x)) $ is a D-open neighborhood  of $ \mathcal{M}\times\{0\} $ in $ \mathcal{M}\times\mathbb{R} $.
	\end{enumerate}
\end{theorem}
\begin{proof}
The proof is inspired by that of usual manifolds  (see, e.g., \cite[Theorem 9.12]{Lee} or \cite{Mic}), but there are technical issues that must be treated.
	Throughout this proof, fix $ x\in\mathcal{M} $ and take any diffeological \'{e}tale chart $ \varphi:U\rightarrow \mathcal{M} $ with $ \varphi(r)=x $  for some $ r\in U $,
	so that $ \Lambda\circ\varphi $ is a plot in $ T^H\mathcal{M} $.
	By Proposition \ref{p-gitm-etl}, the tangent map $ d\varphi $ is a diffeological \'{e}tale map. So for $ (r,d\varphi_r^{-1}(\Lambda(x))) $, there is a local
	lift plot $ \widetilde{\Lambda}:V\rightarrow TU $ defined on an open neighborhood $ V\subseteq U $ of $ r $ such that
	$ d\varphi\circ \widetilde{\Lambda}=\Lambda\circ\varphi|_V $ and $ \widetilde{\Lambda}(r)=(r,d\varphi_r^{-1}(\Lambda(x))) $.
	On the other hand, 	$ \pi_{\mathcal{M}}\circ d\varphi =\varphi\circ\pi_U $.
	Putting these together, we get 
	\begin{center}
		$ \pi_{\mathcal{M}}\circ d\varphi\circ \widetilde{\Lambda}=\varphi\circ\pi_U\circ \widetilde{\Lambda}\qquad $
		or 
		$ \qquad\varphi|_V=\pi_{\mathcal{M}}\circ (\Lambda\circ\varphi|_V)=\varphi\circ\pi_U\circ \widetilde{\Lambda} $,
	\end{center}
    which by Corollary \ref{cor-inj} implies that
	$ \pi_U\circ \widetilde{\Lambda}|_W=\mathrm{id}_W $ for some open neighborhood $ W\subseteq V $ of $ r $, because $ \varphi $ a diffeological injective and $ \pi_U\circ \widetilde{\Lambda}(r)=r $.
	Therefore, $ \widetilde{\Lambda}|_W $ is a vector field on $ W $.
	\\
	\noindent
	\textit{Proof of part} (a).
	\textsc{(Existence)} By existence of ODE solutions, one can find an
	integral curve $ \gamma:(-\epsilon,\epsilon)\rightarrow W, \epsilon>0,  $ for $ \widetilde{\Lambda}|_W $ with $ \gamma(0)=r $. Hence $ \varphi\circ\gamma:(-\epsilon,\epsilon)\rightarrow \mathcal{M}, \epsilon>0,  $ is an
	integral curve of $ \Lambda $ with $ \gamma(0)=x $.
	
	\textsc{(Uniqueness)} Assume that $ \gamma_i:(-\epsilon_i,\epsilon_i)\rightarrow \mathcal{M}, \epsilon_i>0, i=1,2  $ are integral curves of $ \Lambda $ with $ \gamma_i(0)=x $.
	Then for each $ i=1,2 $,
	there is a curve $ \widetilde{\gamma_i}:(-\delta_i,\delta_i)\rightarrow U, 0<\delta_i\leq\epsilon_i $ such that
	$ \varphi\circ \widetilde{\gamma_i}=\gamma_i $ and $ \widetilde{\gamma_i}(0)=r $.
	Taking differential, we get
	\begin{center}
		$ d\varphi \circ \widetilde{\gamma_i}'(t)=\gamma_i'(t)=\Lambda(\gamma_i(t))=\Lambda(\varphi\circ \widetilde{\gamma_i}(t))=d\varphi\circ \widetilde{\Lambda}(\widetilde{\gamma_i}(t)).  $
	\end{center}
	Since $ d\varphi_{\widetilde{\gamma_i}(t)} $ is injective, 
	$ \widetilde{\gamma_i}'(t)= \widetilde{\Lambda}(\widetilde{\gamma_i}(t)) $. That is, 
	$ \widetilde{\gamma_1} $ and $ \widetilde{\gamma_2} $ are integral curves of the vector field
	$ \widetilde{\Lambda}|_W $ on some open neighborhood of $ r $.
	Now by uniqueness of ODE solutions, $ \widetilde{\gamma_1}|_{(-\delta,\delta)} = \widetilde{\gamma_2}|_{(-\delta,\delta)} $ and consequently $ \gamma_1|_{(-\delta,\delta)} = \gamma_2|_{(-\delta,\delta)} $, where $ \delta=\mathrm{min}\{\delta_1,\delta_2\} $.
	
	\textsc{(Maximality)}
	Consider the family of all integral curves of $ \Lambda $ passing through $ x $ at $ 0 $, which is compatible by the uniqueness property above.
	By D3, the supremum of this family is obviously the unique maximal integral
	curve passing through $ x $ at $ 0 $. 
	
 
\noindent
	\textit{Proof of part} (b).
		\textsc{(Group laws)}
	Analogous to the usual case, 
	one can  check the group laws for $ Fl $.
	
	\textsc{(Smoothness)}
	By smoothness part of the ODE theorem (see, e.g., \cite[Theorem D.1]{Lee}), 
	for the vector field $ \widetilde{\Lambda}|_W $ on $ W $,
	we get a smooth map
	$ \theta:W'\rightarrow\mathrm{C}^{\infty}((-\epsilon,\epsilon), W)$ with $\epsilon>0 $ defined on an open neighborhood $ W'\subseteq W $ of $ r $  
	such that
	$ \theta(w)(0)=w $ and $ \widetilde{\Lambda}(\theta(w)(t))=(\theta(w))'(t) $, 
	for all $ w\in W' $ and $ t\in (-\epsilon,\epsilon) $.
	Define
	\begin{center}
		$ \Theta:W'\rightarrow\mathrm{C}^{\infty}((-\epsilon,\epsilon), \mathcal{M})\qquad $ by
		$ \qquad\Theta(w)(t)=\varphi\circ\theta(w)(t) $, 
	\end{center}
	which is smooth by \cite[\S 1.59]{PIZ}.
	We observe that 
	$ \Theta(w)(0)=\varphi\circ\theta(w)(0)=\varphi(w) $ and that
	\begin{center}
		$ \Theta(w)'(t)=d\varphi\circ(\theta(w))'(t)=d\varphi\circ \widetilde{\Lambda}(\theta(w)(t))=\Lambda\circ\varphi(\theta(w)(t))=\Lambda(\Theta(w)(t)). $
	\end{center}
	By uniqueness, $ \Theta(w) $ 
	is equal to $ Fl\circ\varphi(w)|_{(-\epsilon,\epsilon)} $ and so
	$ W'\times(-\epsilon,\epsilon)\rightarrow \mathcal{M}, (w,t)\mapsto Fl\circ \varphi(w)(t) $ is 
	a plot in $ \mathcal{M} $.
	Set $ O:=\varphi(W') $, which is a D-open neighborhood of $ x $ in $ \mathcal{M} $. 
	Because $ \{\varphi|_{W'}\} $ is a covering generating family for $ O $, the map $  O\times(-\epsilon,\epsilon)\rightarrow \mathcal{M} $ given by
	$(z,t)\mapsto Fl (z)(t) $ is smooth, too.
	
	Let $ I_{x} $ be the set of all $ s\in \mathbb{R} $
	such that the map $(z,t)\mapsto Fl(z)(t) $ is defined and smooth on a D-open neighborhood of $ \{x\}\times[\min\{0,s\}, \max\{0,s\}] $ in $ \mathcal{M}\times\mathbb{R} $. The theorem will be proved if we can show that $ I_{x}=\mathrm{dom}(Fl(x)) $.
	By the argument above, $ 0\in I_{x}\neq\varnothing $. 
	Moreover, by definition, $ I_{x}\subseteq \mathrm{dom}(Fl(x)) $ is open.
	Now we prove that $ I_{x} $ is closed in $ \mathrm{dom}(Fl(x)) $ so that $ I_{x}=\mathrm{dom}(Fl(x)) $.
	Suppose on the contrary that there exists an element $ s_0 $ in $ \mathrm{dom}(Fl(x))\cap(\overline{I_{x}}\setminus I_{x})\neq\varnothing $.
	We can assume that $ s_0>0 $, say.
	Let $ y:=Fl(x)(s_0) $. Again by an argument similar to above, 
	we can find a D-open neighborhood $ O_{y} $ of $ y $ in $ \mathcal{M} $ and a sufficiently small number $ 0<\epsilon<s_0 $ such that the map $  O_{y}\times(-\epsilon,\epsilon)\rightarrow \mathcal{M} $ given by
	$(z,t)\mapsto Fl (z)(t) $ is smooth.
	Since $ s_0\in \overline{I_{x}}\setminus I_{x} $, there exists an $ s_1\in I_{x}~~~\cap~~~(s_0-\epsilon,s_0+\epsilon)\cap(Fl (x)^{-1}(O_{y})) $, which means that
	 $(z,t)\mapsto Fl(z)(t) $ is defined and smooth on a D-open neighborhood $  O_{x}\times(-\delta,s_1+\delta) $ of $ \{x\}\times[0,s_1] $ in $ \mathcal{M}\times\mathbb{R} $, and $ Fl(x)(s_1)\in O_{y} $.
	 As the map $ z\mapsto Fl(z)(s_1) $ is smooth, we can choose $ O_{x} $ small enough that $ z\mapsto Fl(z)(s_1) $ maps $ O_{x} $ into $ O_{y} $.
	 Now define $ O_{x}\times(-\delta,s_1+\epsilon)\rightarrow\mathcal{M} $ by
	 	\begin{align*}
	 	(z,t)\longmapsto\left \{
	 	\begin{array}{lr}
	 		Fl(z)(t),\quad &  -\delta <t < s_1\\
	 		Fl(Fl(z)(s_1))(t-s_1), \quad & s_1-\epsilon< t<s_1+\epsilon\\
	 	\end{array} \right.
	 \end{align*}
 which is well-defined by the group laws property and smooth from the fact that $(z,t)\mapsto Fl(z)(t) $ is defined and smooth on both $  O_{x}\times(-\delta,s_1+\delta) $ and $  O_{y}\times(-\epsilon,\epsilon) $.
But $ O_{x}\times(-\delta,s_1+\epsilon) $ is a D-open neighborhood of $ \{x\}\times[0, s_0] $ so that  $ s_0\in I_{x} $, a contradiction! 
	\end{proof}

\begin{definition} 
	We say that a vector field $ \Lambda:\mathcal{M}\rightarrow T^H\mathcal{M} $ on a  diffeological \'{e}tale manifold $ \mathcal{M} $  is \textbf{complete} 
	if
	each of its maximal integral curves is defined for all $ t\in\mathbb{R} $, or equivalently 
	if it generates a complete flow.
\end{definition}
\begin{proposition}\label{pro-comp}
	Any vector field $ \Lambda:\mathcal{M}\rightarrow T\mathcal{M} $ on a D-compact diffeological \'{e}tale manifold $ \mathcal{M} $  is complete.
\end{proposition}
\begin{proof}
	The proof is similar to the classical one (see, e.g., \cite[Lemma \S 3.8]{Mic}).
\end{proof}
As an application of this discussion, we are able to compute the internal tangent spaces of the diffeomorphism group of a D-compact diffeological \'etale manifold.
\begin{proposition}
	Suppose that $ \mathcal{M} $ is a diffeological \'{e}tale manifold. Then the subspace diffeology  inherited from $ \mathrm{C}^{\infty}(\mathcal{M},\mathcal{M}) $ and the standard diffeology on $ \mathrm{Diff}(\mathcal{M}) $ (see Example \ref{exa-diff}) coincide.
	In other words, $ \mathrm{Diff}(\mathcal{M}) $ with the subspace diffeology is a diffeological group.
\end{proposition}
\begin{proof}
	We  must show that any plot in $ \mathrm{Diff}(\mathcal{M}) $ for the subspace diffeology is a plot for the standard diffeology.
	Let $ P : U \rightarrow\mathrm{Diff}(X) $ be a plot for the subspace diffeology.
	The map 
\begin{center}
		$ \tilde{P}: U\times\mathcal{M}\rightarrow \mathcal{M},\quad (r,x)\mapsto P(r)(x) $
\end{center}
	is smooth and $ d\big{(}\tilde{P}^{r}\big{)}_{x}= dP(r)_{x} $ is an isomorphism, for all $ (r,x)\in U\times\mathcal{M} $. By Theorem \ref{the-ImFT}, then 
	\begin{center}
		$ (\Pr_1,\tilde{P}):U\times\mathcal{M}\rightarrow U\times\mathcal{M},\quad (r,x)\mapsto (r,P(r)(x)) $
	\end{center}
	is a diffeological \'{e}tale map.
	Moreover, it is bijective with the inverse 
	\begin{center}
		$  U\times\mathcal{M}\rightarrow U\times\mathcal{M},\quad (r,x)\mapsto (r,P(r)^{-1}(x)) $.
	\end{center}
	and consequently, a diffeomorphism. Therefore, the map
	$  U\times\mathcal{M}\rightarrow \mathcal{M}, (r,x)\mapsto P(r)^{-1}(x) $ is smooth and
	the parametrization $ r\mapsto P(r)^{-1} $ is a plot as desired.
\end{proof}

\begin{proposition}
	For a D-compact diffeological \'{e}tale manifold $ \mathcal{M} $, 
	the map $ \Upsilon:T_{\mathrm{id}_{\mathcal{M}}}\mathrm{Diff}(\mathcal{M})\rightarrow\mathfrak{X}(\mathcal{M}) $ taking $ dP_0(\frac{d}{dt}(0)) $ to $ \alpha\circ\widehat{P} $ is an isomorphism, where $ \alpha $ is given in Lemma \ref{lem-1} and $ \widehat{P}:\mathcal{M}\rightarrow \mathrm{C}^{\infty}(\mathbb{R},\mathcal{M})$ is defined by $ \widehat{P}(x)(s)=P(s)(x) $, for all $ x\in \mathcal{M}, s\in \mathbb{R} $.
\end{proposition}
\begin{proof}
	The proof is entirely analogous to that of \cite[Proposition 6.3]{HM-V} and \cite[\S 4.5]{Hec}.
\end{proof}

\begin{corollary}\label{coro-diff-comp}
	For a D-compact diffeological \'{e}tale manifold $ \mathcal{M} $, 
	$ T_{f}\mathrm{Diff}(\mathcal{M}) $	is isomorphic to $ \mathfrak{X}(\mathcal{M}) $, for every $ f\in \mathrm{Diff}(\mathcal{M}) $.
\end{corollary}

\begin{example}
	$ T_{f}\mathrm{Diff}(\mathbb{T}_{\alpha}) $	is isomorphic to $ \mathfrak{X}(\mathbb{T}_{\alpha}) $, for every $ f\in \mathrm{Diff}(\mathbb{T}_{\alpha}) $.
\end{example}


\end{document}